\documentclass[11pt]{article}

\usepackage{amsmath,amssymb,amsthm}
\usepackage{mathtools}
\usepackage{bm}
\usepackage{geometry}
\usepackage{enumerate}
\usepackage{cite}

\newtheorem{theorem}{Theorem}[section]
\newtheorem{proposition}[theorem]{Proposition}

\newtheorem{corollary}[theorem]{Corollary}

\theoremstyle{definition}

\theoremstyle{remark}
\newtheorem{remark}[theorem]{Remark}

\newcommand{\ep}{\varepsilon}
\newcommand{\RR}{\mathbb{R}}
\newcommand{\CC}{\mathbb{C}}
\newcommand{\TT}{\mathbb{T}}
\newcommand{\p}{\partial}

\DeclareMathOperator{\diag}{diag}

\title{Local well-posedness for fourth-order nonlinear 
\\
dispersive systems on the one-dimensional torus 
\\
under structural conditions}
\author{
	Eiji Onodera\vspace{0.5em}\\
	Department of Mathematics and Physics, 
	\\Faculty of Science and Technology,  
	\\Kochi University, Kochi 780-8520, Japan\\
	\texttt{Email:onodera@kochi-u.ac.jp}
}
\date{}

\begin{document}

\maketitle

\begin{abstract}
We study local well-posedness of the initial value problem 
for a class of fourth-order nonlinear dispersive systems on the one-dimensional torus. 
The main difficulty comes from the loss of derivatives in the nonlinear terms. 
By introducing suitable structural conditions 
that compensate for derivative loss 
through cancellation mechanisms, 
and constructing modified energies 
via gauge-type transformations combined with 
a diagonalization procedure,
we establish local well-posedness in high-regularity Sobolev spaces.
The assumptions in the present result relax those in previous works,  
at the expense of requiring higher regularity. 
In particular, our structural conditions extend the range of admissible nonlinearities 
beyond the scalar case and provide a unified framework for treating multi-component systems.
\end{abstract}





\section{Introduction}
\label{sec:introduction}
We consider the following initial value problem 
\begin{equation}
\label{eq:IVP}
\begin{cases}
\displaystyle
\p_t Q - i M \p_x^4 Q - M_S \p_x^2 Q
= F(Q, \p_x Q, \p_x^2 Q), 
\\
Q(0,\cdot) = Q_0(\cdot)\in H^m(\TT;\CC^n),
\end{cases}
\end{equation}
where 
$Q = (Q_1, \dots, Q_n)^{\top}(t,x) : \mathbb{R} \times \mathbb{T} \to \mathbb{C}^n$ 
and $Q _0= (Q_{01}, \dots, Q_{0n})^{\top}(x): \mathbb{T} \to \mathbb{C}^n$  
are respectively 
an unknown function 
and a given initial function,  
$\mathbb{T} = \mathbb{R}/2\pi\mathbb{Z}$ denotes 
the one-dimensional torus, 
$i=\sqrt{-1}$ denotes the imaginary unit, 
\[
M=\diag(a_1, \ldots,a_n)
\quad
\text{for}
\quad  
(a_1,\ldots,a_n) \in (\mathbb{R} \setminus \{0\})^n,
\] 
and $M_S \in \mathbb{C}^{n \times n}$ is a constant skew-Hermitian matrix, 
i.e., $M_S^* = - M_S$. 
The nonlinear term 
$F(Q, \p_x Q, \p_x^2 Q)
= (F_1, \dots, F_n)^{\top}(Q, \p_x Q, \p_x^2 Q)$ 
depends on $Q$ and its derivatives up to second order, and 
is defined componentwise as follows: 
\begin{align}
\nonumber
F_j(Q, \p_x Q, \p_x^2 Q)
&=  
\sum_{p,q,r=1}^n
\omega^{1,j}_{p,q,r}
\, \p_x^2 Q_p
\, \overline{Q_q}
\, Q_r
+ 
\sum_{p,q,r=1}^n
\omega^{2,j}_{p,q,r}
\, \overline{\p_x^2Q_p}
\, Q_q
\, Q_r \notag\\
&\quad
+ 
\sum_{p,q,r=1}^n
\omega^{3,j}_{p,q,r}
\, \p_x Q_p
\, \overline{\p_x Q_q}
\, Q_r
+ 
\sum_{p,q,r=1}^n
\omega^{4,j}_{p,q,r}
\, \p_x Q_p
\, \p_xQ_q
\, \overline{Q_r}
\notag\\
&\quad
+O\left(\sum_{k=1}^{\text{finite}}|Q|^k\right)
\qquad 
(j\in \{1, \dots, n\}). 
\nonumber
\end{align}
Here, all coefficients 
$\omega^{k,j}_{p,q,r} \in \mathbb{C}$ 
for $p,q,r,j\in \{1,\ldots, n\}$ and $k\in \{1,\ldots,4\}$
are complex constants. 
For $m \in \mathbb{N}$, 
we denote by $H^m(\mathbb{T}; \mathbb{C}^n)$ 
the standard Sobolev space
of $\mathbb{C}^n$-valued functions on $\mathbb{T}$ 
with norm and $L^2$-inner product 
\begin{align}
\| f \|_{H^m}
&=\sqrt{\sum_{k=0}^m \| \p_x^k f \|_{L^2}^2} 
=\sqrt{
\sum_{k=0}^m \langle \p_x^k f, \p_x^k f\rangle
},
\nonumber
\\
\langle
f,g
\rangle
&:=
\int_{\TT} f(x)\cdot g(x) \,dx 
=
\sum_{j=1}^n
\int_{0}^{2\pi}f_j(x)\overline{g_j(x)}\,dx 
\nonumber
\end{align}
for   
$f=(f_1,\ldots,f_n)^{\top}$ and  
$g=(g_1,\ldots,g_n)^{\top}$ in $L^2(\TT;\mathbb{C}^n)$. 
\par 
The purpose of this paper is to present 
local well-posedness results for the initial value 
problem \eqref{eq:IVP} 
in high-regularity Sobolev spaces $H^m(\TT;\CC^n)$.
We emphasize that the presence of second- and first-order derivatives in the nonlinear terms 
causes the so-called loss of derivatives: 
Applying $\p_x^m$ to the system leads to terms involving derivatives of order $m+2$ or $m+1$, 
which cannot be controlled directly by the classical energy method based on the $H^m$ norm.
This is the main analytical difficulty in the study 
of \eqref{eq:IVP}, and special structural properties 
are required in order to close the energy estimates.
In addition, compared with the scalar case ($n=1$), 
the analysis of systems requires a more careful treatment 
of the interaction between different components.
The key idea of this paper is that the derivative loss can be compensated 
by exploiting structural conditions that eliminate certain second-order terms 
and transform the remaining problematic terms into divergence-type forms, 
allowing cancellation in the energy estimates.
Accordingly, the aim of this paper is to identify such conditions in a systematic way.
\par 
To state our results, 
we introduce structural conditions on the nonlinear terms.
We first comment on symmetry conditions (A1)--(A2) on the coefficients of $F$: 
\begin{align}
\tag{A1}
\omega^{2,j}_{p,q,r} 
= 
\omega^{2,j}_{p,r,q} 
\quad \text{for all} \quad p,q,r,j\in 
\{1,\ldots,n\}, 
\\
\tag{A2}
\omega^{4,j}_{p,q,r} 
= 
\omega^{4,j}_{q,p,r} 
\quad \text{for all} \quad p,q,r,j\in 
\{1,\ldots,n\}.
\end{align}
These conditions correspond to symmetric rearrangements of the nonlinear terms 
$$\sum_{p,q,r=1}^n
\omega^{2,j}_{p,q,r}\, \overline{\p_x^2Q_p}\, Q_q\, Q_r
\quad \text{and} \quad 
\sum_{p,q,r=1}^n
\omega^{4,j}_{p,q,r}\, \p_x Q_p\, \p_xQ_q\, \overline{Q_r},
$$
and do not restrict the generality of the formulation even if we impose them. 
We note that this paper will impose only (A2) to state Theorems \ref{thm:main} and \ref{thm:main2},  
because it turns out that 
any conditions on $\omega_{p,q,r}^{2,j}$ 
are not required to state these theorems. 
We next impose additional essential conditions (C1)-(C2):  
\begin{align}
\tag{C1}
\omega^{1,j}_{r,q,p} 
&=
-\overline{\omega^{1,r}_{j,p,q} }
\quad \text{for all} \quad p,q,r,j\in 
\{1,\ldots,n\}, 
\\
\tag{C2}
\omega^{3,j}_{r,q,p} 
-2 \overline{\omega^{4,r}_{j,q,p}}
&=
-(\overline{\omega^{3,r}_{j,p,q}} 
- 2 \omega^{4,j}_{r,p,q})
\quad \text{for all} \quad p,q,r,j
\in \{1,\ldots,n\}. 
\end{align}
Roughly speaking, (C1) eliminates the most problematic second-order terms, while (C2) reduces 
the remaining problematic first-order terms to 
divergence-type forms. 
These conditions can be viewed as a structural 
characterization of the mechanism that prevents 
the loss of derivatives in this class of systems.
\par 
We are now in a position to state the main result of this paper.
\begin{theorem}
\label{thm:main}
Assume that $M=a\,I_n$ for some 
$a\in \mathbb{R} \setminus \{0\}$ 
where $I_n$ denotes the identity matrix of order $n$ 
and $F(Q,\p_xQ,\p_x^2Q)$ satisfies 
all the conditions 
\emph{(A2)} and \emph{(C1)--(C2)}. 
Then the initial value problem \eqref{eq:IVP}  
 is time-locally well-posed in $H^m(\TT;\mathbb{C}^n)$ 
 for any integer $m\geqslant 5$ 
 in the following sense:
 \begin{enumerate}
 \item[\textup{(i)}] (Existence and uniqueness.) 
 For any $Q_0\in H^m(\TT;\mathbb{C}^n)$, there exists 
 $T=T(\|Q_0\|_{H^5})>0$ and a unique solution 
 $Q\in C([-T,T]; H^{m}(\TT;\mathbb{C}^n))$
 to \eqref{eq:IVP}.
 \item[\textup{(ii)}] (Continuous dependence with respect to the initial data in $H^m$-topology.) 
 Suppose that $T>0$ and $Q\in C([-T,T];H^m(\TT;\mathbb{C}^n))$ are respectively 
 the time and the unique solution 
 to \eqref{eq:IVP} with initial data $Q_0$ obtained in the above part \textup{(i)}. 
 Fix  $T^{\prime}\in (0,T)$. 
 Then for any $\eta>0$, there exists $\delta>0$ such that for any 
 $\widetilde{Q_0}\in H^m(\TT;\mathbb{C}^n)$ satisfying 
 $\|Q_0-\widetilde{Q}_0\|_{H^m}<\delta$, 
 the unique solution $\widetilde{Q}$ 
 to \eqref{eq:IVP} with initial data replaced by $\widetilde{Q}_0$ exists on 
 $[-T^{\prime},T^{\prime}]\times \TT$ and 
 satisfies 
 $\|Q-\widetilde{Q}\|_{C([-T^{\prime},T^{\prime}];H^m)}<\eta$.
 \end{enumerate}
\end{theorem}
We state known results on local well-posedness on $\TT$. 
	The scalar case ($n=1$) was studied by 
	Segata~\cite{Segata2012}, 
	who established local well-posedness in $H^m(\TT;\CC)$ 
	for integers $m\geqslant 4$ under the assumption that  
	$\omega_{1,1,1}^{1,1}, \omega_{1,1,1}^{2,1},\omega_{1,1,1}^{3,1},
	\omega_{1,1,1}^{4,1}\in i\,\RR$. 
	The result was extended to the system case 
	by the present author~\cite{Onodera2025}, 
	who proved local well-posedness in $H^m(\mathbb{T}; \mathbb{C}^n)$ 
	for integers $m \geqslant 4$ and $n\geqslant 1$ 
	under the conditions (A1)--(A2) 
	and additional conditions (B1)--(B6): 
\begin{align}
	\tag{B1} 
	\omega_{p,q,r}^{1,j}&=\omega_{r,q,p}^{1,j} 
	\quad \text{for all} \quad p,q,r,j\in 
	\{1,\ldots,n\}, 
\\
	\tag{B2} 
	\omega_{p,q,r}^{1,j}&=-\overline{\omega_{j,r,q}^{1,p}} 
	\quad \text{for all} \quad p,q,r,j\in 
	\{1,\ldots,n\},  
\\
	\tag{B3} \omega_{p,q,r}^{2,j}&=-\overline{\omega_{r,j,p}^{2,q}}
	\quad \text{for all} \quad p,q,r,j\in 
	\{1,\ldots,n\},  
\\
	\tag{B4} \omega_{p,q,r}^{3,j}&=\omega_{r,q,p}^{3,j} 
	\quad \text{for all} \quad p,q,r,j\in 
	\{1,\ldots,n\},  
\\
	\tag{B5}  \omega_{p,q,r}^{3,j}&=-\overline{\omega_{j,r,q}^{3,p}} 
	\quad \text{for all} \quad p,q,r,j\in 
	\{1,\ldots,n\},  
\\
	\tag{B6} \omega_{p,q,r}^{4,j}&=-\overline{\omega_{j,r,q}^{4,p}} 
	\quad \text{for all} \quad p,q,r,j\in 
	\{1,\ldots,n\}. 
	\nonumber
\end{align}  
In the scalar case, the set of conditions 
(A1)--(A2) and (B1)--(B6) is equivalent to the following 
condition:  
\[
\omega_{1,1,1}^{1,1}, 
\omega_{1,1,1}^{2,1},
\omega_{1,1,1}^{3,1},
\omega_{1,1,1}^{4,1}\in i\,\RR. 
\] 
Hence, no additional solvable structure was obtained 
in the scalar case in \cite{Onodera2025}. 
\begin{remark}
\label{rem:6211}
	The structural conditions (A1)--(A2) and (B1)--(B6) arise naturally in several contexts.
	From a geometric viewpoint, these conditions are related to the structure of the Riemann curvature tensor 
	on locally Hermitian symmetric spaces $N$ 
	of complex dimension $n$, 
	which appears in the formulation of generalized bi-Schr\"odinger flow equations for curve flows on $N$. 
	To be more concretely, 
	the geometric equation was introduced by Ding and 
	Wang in \cite{DW2018}, 
	and it was proved in \cite{Onodera2019,Onodera2022}
	that the geometric structure of $N$ ensures 
	time-local solvability of the initial value problem. 
	It was also shown in \cite{Onodera2024} 
	that the solvable structure is formally inherited 
	to our $n$-component system as the 
	structural conditions (A1)--(A2) and (B1)--(B6) 
	via the generalized Hasimoto transformation. 
	From a physical perspective, these conditions 
	are also satisfied by models such as the 
	vortex filament equation \cite{fukumoto,FM} 
	and  systems describing the interaction of multiple pulses in optical fibers~\cite{WZY}. 
	For further details, see \cite{Onodera2025}.
\end{remark}
\par 
In contrast, the set of conditions (A2) and (C1)--(C2) 
in Theorem~\ref{thm:main} 
is weaker than that of conditions 
(A1)--(A2) and (B1)--(B6) in the previous results 
in \cite{Onodera2025}. 
Indeed, 
the conditions (C1)--(C2) are strictly weaker than (B1)--(B6), 
but the converse does not hold in general. 
In the scalar case, 
the set of conditions (A2) and (C1)--(C2) 
is equivalent to the following condition: 
\[
\omega_{1,1,1}^{1,1}\in i\,\RR 
\quad 
\text{and}
\quad 
\operatorname{Re}[\omega_{1,1,1}^{3,1}
-2\omega_{1,1,1}^{4,1}]=0,
\]
allowing non-purely imaginary complex numbers for 
$\omega_{1,1,1}^{2,1},
\omega_{1,1,1}^{3,1}$ and 
$\omega_{1,1,1}^{4,1}$. 
Thus, the present result also extends 
the scalar result of Segata \cite{Segata2012}.
Although this relaxation is already present in the scalar case, 
it becomes more apparent in the system case $n \ge 2$.
Therefore, the assumptions in Theorem~\ref{thm:main} relax those in previous works 
at the cost of requiring a higher regularity index $m \ge 5$. 
\par 
Further, we can present additional conditions (D1)--(D2) 
with (A1) to avoid the use of $m\geqslant 5$ 
in the proof of Theorem~\ref{thm:main}, 
ensuring local well-posedness in 
$H^m(\TT;\CC)$ for $m\geqslant 4$: 
\begin{align}
	\tag{D1} 
	\omega_{r,q,p}^{2,j}=\omega_{j,q,p}^{2,r} 
	\quad \text{for all} \quad p,q,r,j\in 
	\{1,\ldots,n\},  
\\
		\tag{D2} 
		\omega_{p,r,q}^{3,j}=\omega_{p,j,q}^{3,r} 
		\quad \text{for all} \quad p,q,r,j\in 
		\{1,\ldots,n\}.
\end{align}
\begin{corollary}
\label{cor:cor}
Assume that $M=aI_n$ for some 
$a\in \mathbb{R} \setminus \{0\}$
and $F(Q,\p_xQ,\p_x^2Q)$ satisfies 
all the conditions 
\emph{(A1)--(A2)}, \emph{(C1)--(C2)} and \emph{(D1)--(D2)}. 
Then \eqref{eq:IVP}  
 is time-locally well-posed in $H^m(\TT;\mathbb{C}^n)$ 
 for any integer $m\geqslant 4$.
\end{corollary}
This also extends the previous results in \cite{Onodera2025} 
on local well-posedness in $H^m(\TT;\CC^n)$ for 
$m\geqslant 4$  under 
(A1)--(A2) and (B1)--(B6). 
Indeed, (D1) follows from (A1) and (B3), 
and (D2) follows from (B4) and (B5), 
but their converses do not hold in general. 
\par 
Our idea of the proof of Theorem~\ref{thm:main}
and Corollary~\ref{cor:cor}
is the combination of the diagonalization technique 
and  estimates for modified energies 
via some transformations of gauge-type:  
We first derive $n$-component systems 
satisfied by higher-order derivatives $U=\p_x^mQ$
and extend them to 
$2n$-component systems for 
$(U\ \overline{U})^{\top}$. 
We then perform a diagonalization procedure to 
eliminate seemingly bad lower-order terms 
in block-off-diagonal part  
by introducing a transformation of gauge-type  
which is formally defined as a sum of the identity and 
two pseudodifferential operators of order $-2$ and 
$-3$. 
\par 
Within this framework, 
at the level of higher-order derivatives, terms involving 
$\overline{\p_x^{m+1}U}$ and $\overline{\p_x^{m+2}U}$ 
may also appear to lead to a loss of derivatives 
when viewed within the classical energy method. 
However, these terms are in fact harmless due to 
their conjugate structure. 
The diagonalization procedure clarifies this point 
by reformulating the system so that 
such contributions appear in the block-off-diagonal part.
\par 
The remaining terms responsible for the derivative loss
after diagonalization are treated 
by an additional transformation of gauge-type.  
The structural conditions on the nonlinear terms 
play a crucial role in this step: 
(C1) eliminates the Hermitian part of the 
second-order 
diagonal terms and (C2) converts 
the skew-Hermitian part 
of the first-order diagonal terms 
into divergence-type forms. 
This structure enables cancellation of 
the problematic contributions 
in the energy estimates, 
thereby leading to the construction of 
suitable modified energies to close the estimates 
at the $H^m$ level.
\par 
Although the required gauge-type transformations 
can in principle be identified directly, 
without resorting to the extension to a $2n$-component 
system or a diagonalization procedure, 
the diagonalization approach provides a systematic 
and transparent way to determine them 
and to reveal the underlying mechanism of derivative loss.
\par 
More concretely, 
our proof of Theorem~\ref{thm:main} 
and Corollary~\ref{cor:cor}
is based on the above idea and
a fourth-order parabolic regularization of the 
Bona-Smith type \cite{BS}, 
the latter of which is a standard tool to show the 
continuous dependence of the solution with respect to
the initial data in $H^m$-topology. 
Uniqueness of the solution 
is proved by estimating an $H^2$-energy for the difference 
of two solutions 
with gauge-type transformations implicitly, 
where we explicitly adopt a trick of 
the modified energy method 
established by Kwon \cite{Kwon} and Segata \cite{Segata2012}.  
\par 
The assumptions $m\geqslant 5$ in Theorem~\ref{thm:main} and 
$m\geqslant 4$ in Corollary~\ref{cor:cor}  come from the requirement 
on the gauge-type transformations to work. 
In particular, 
the assumption $m\geqslant 5$ comes from the 
requirement that the pseudodifferential operator of order $-3$
to form the gauge-type transformation 
works to eliminate the 
seemingly bad first-order terms with skew-Hermitian 
coefficient matrix in block-off-diagonal part. 
It is to be noted that 
the use of the pseudodifferential operator of order $-3$ 
is not required to prove Corollary~\ref{cor:cor} 
and the previous local well-posedness 
results \cite{Onodera2025} 
in $H^m(\TT;\CC^n)$ for $m\geqslant 4$, 
by the absence of such bad lower order terms.  
\par 
We mention that the above idea and the choice 
of the gauge-type transformations 
are inspired by the $L^2(\TT)$-well-posedness 
theory for 
linear fourth-order dispersive equations for 
$\CC^n$-valued functions established 
by Mizuhara \cite{Mizuhara} ($n=1$) and 
Chihara \cite{Chihara2015} ($n=2$).  
In addition, the diagonalization technique was originated 
by Chihara \cite{Chihara1995} to study a class of 
Schr\"odinger-type equations 
with nonlinear terms including the complex conjugate of the first-order spatial derivative of the unknown function. 
The technique of observing the structure of equations 
by coupling their complex conjugate 
has been also applied or discussed 
in some situations for 
nonlinear Schr\"odinger-type equations 
and for more general nonlinear second-order dispersive type equations. See, e.g., \cite{Chihara1995_2, 
KPV1998, KPV2004} and references therein.   
\par 
Returning to the nonlinear problem \eqref{eq:IVP}, 
we note that 
the assumption $M=aI_n$ plays a role to ensure that 
each of the gauge-type transformation and the principal part $i\,M\p_x^4$ 
commute with each other. 
Without the assumption,  
any nonvanishing commutator generally causes a loss of either one or two derivatives, 
which prevents our proof strategy from working. 
However, it is turned out that, 
even if $M=\diag(a_1,\ldots,a_n)$ for 
$(a_1,\ldots,a_n) \in (\mathbb{R} \setminus \{0\})^n$ 
without $a_1=\cdots=a_n$, 
our proof strategy remains valid under the additional conditions (E1)-(E3): 
\begin{align}
	\tag{E1}
	\omega^{2,j}_{r,q,p} =- \omega^{2,r}_{j,q,p} 
	\quad &\text{unless $j=r$ for all} \quad p,q,r,j\in 
	\left\{1,\ldots,n\right\}, 
	\\
	\tag{E2}
	\omega^{1,j}_{r,q,p}
	=
	\overline{
		\omega^{1,r}_{j,p,q}
	}
	\quad &\text{unless $j=r$ for all} \quad p,q,r,j\in 
	\{1,\ldots,n\},
	\\
	\tag{E3}
	\omega^{3,j}_{r,q,p}
	-2\overline{\omega^{4,r}_{j,q,p}}
	=0
	\quad &\text{unless $j=r$ for all} \quad p,q,r,j\in 
	\{1,\ldots,n\}.
\end{align}
\begin{corollary}
\label{cor:cor2}
Assume that $M=\diag(a_1,\ldots,a_n)$ for some  
$(a_1,\ldots,a_n) \in (\mathbb{R} \setminus \{0\})^n$ 
and $F(Q,\p_xQ,\p_x^2Q)$ satisfies 
all the conditions  
\emph{(A1)--(A2)}, \emph{(C1)--(C2)}, 
\emph{(D1)--(D2)} and \emph{(E1)-(E3)}. 
Then \eqref{eq:IVP}  
 is time-locally well-posed in $H^m(\TT;\mathbb{C}^n)$ 
 for any integer $m\geqslant 4$.
\end{corollary}
Roughly speaking, these conditions ensure 
that all the coefficient part of the gauge-type transformations 
are nothing but entry-wise diagonal matrix-valued functions  
whose off-diagonal entries all vanish.  
We note that a sufficient condition 
based on the same strategy 
has been discussed in \cite{Onodera2025}. 
The above combined conditions slightly extend the sufficient condition in \cite{Onodera2025}. 
However, it seems that the analysis of the case 
$M\neq a\,I_n$ is still 
far reaching. 
\par 
This paper mainly focuses on nonlinearities of the form $F$ for simplicity and for considering geometric or physical background 
mentioned in Remark~\ref{rem:6211}. 
However, we add a comment that the gauge invariance of the 
nonlinear terms is not essential. 
Indeed, the same arguments can be extended to 
more general nonlinear terms 
including 
$G(Q, \p_x Q, \p_x^2 Q)= (G_1, \dots, G_n)^{\top}(Q, \p_x Q, \p_x^2 Q)$ whose $j$-th component is defined by 
\begin{align}
	G_j(Q, \p_x Q, \p_x^2 Q)
	&=  
	\sum_{p,q,r=1}^n
	\mu^{1,j}_{p,q,r}\, \p_x^2 Q_p\, Q_q\, Q_r
	+ 
	\sum_{p,q,r=1}^n
	\mu^{2,j}_{p,q,r}\, \p_x Q_p\, \p_x Q_q\, Q_r
	\notag\\
	&\quad 
	+\sum_{p,q,r=1}^n
	\mu^{3,j}_{p,q,r}\, \p_x^2 Q_p\, \overline{Q_q}\, \overline{Q_r}
	+\sum_{p,q,r=1}^n
	\mu^{4,j}_{p,q,r}\, \overline{\p_x^2 Q_p}\, \overline{Q_q}\, Q_r
	\notag\\
	&\quad
	+ 
	\sum_{p,q,r=1}^n
	\mu^{5,j}_{p,q,r}\, \overline{\p_x Q_p}\, \overline{\p_x Q_q}\, Q_r
	+ 
	\sum_{p,q,r=1}^n
	\mu^{6,j}_{p,q,r}\, \p_x Q_p\, \overline{\p_xQ_q}\, \overline{Q_r}
	\notag\\
	&\quad 
	+\sum_{p,q,r=1}^n
	\mu^{7,j}_{p,q,r}\, \overline{\p_x^2 Q_p}\, \overline{Q_q}\, \overline{Q_r}
	+
	\sum_{p,q,r=1}^n
	\mu^{8,j}_{p,q,r}\, \overline{\p_x Q_p}\, \overline{\p_x Q_q}\, \overline{Q_r}, 
	\nonumber
\end{align}
where all coefficients 
$\mu^{k,j}_{p,q,r} \in \mathbb{C}$ 
for $p,q,r,j\in \{1,\ldots, n\}$ and $k\in \{1,\ldots,8\}$
are complex constants.
To state results for the corresponding initial value problem, 
we introduce the conditions 
\begin{align}
\tag{A3}
\mu^{1,j}_{p,q,r} = \mu^{1,j}_{p,r,q} 
\quad \text{for all} \quad p,q,r,j\in 
\{1,\ldots,n\}, 
\\
\tag{A4}
\mu^{2,j}_{p,q,r} = \mu^{2,j}_{q,p,r} 
\quad \text{for all} \quad p,q,r,j\in 
\{1,\ldots,n\}, 
\\
\tag{A5}
\mu^{3,j}_{p,q,r} = \mu^{3,j}_{p,r,q} 
\quad \text{for all} \quad p,q,r,j\in 
\{1,\ldots,n\}, 
\end{align}
which do not restrict the generality even if we impose. 
Furthermore we impose additional essential conditions 
\begin{align}
\tag{C3}
\mu^{1,j}_{r,q,p} 
&=-\overline{\mu^{3,r}_{j,q,p}} 
\quad \text{for all} \quad p,q,r,j\in 
\{1,\ldots,n\}, 
\\
\tag{C4}
2\mu^{2,j}_{r,q,p} 
-\overline{
\mu_{j,q,p}^{6,r}
}
&= 
2\mu^{2,j}_{r,p,q} 
-\overline{
\mu_{j,p,q}^{6,r}
}
\quad \text{for all} \quad p,q,r,j\in 
\{1,\ldots,n\}.
\end{align}
\begin{theorem}
\label{thm:main2}
Assume that $M=aI_n$ for some 
$a\in \mathbb{R} \setminus \{0\}$, 
$F(Q,\p_xQ,\p_x^2Q)$ satisfies 
all the conditions 
\emph{(A2)} and \emph{(C1)--(C2)}, 
and that $G(Q,\p_xQ,\p_x^2Q)$ satisfies 
all the conditions 
\emph{(A3)--(A5)} and \emph{(C3)--(C4)}.  
Then \eqref{eq:IVP} 
with $F(Q,\p_xQ,\p_x^2Q)$ replaced by 
$F(Q,\p_xQ,\p_x^2Q)+G(Q,\p_xQ,\p_x^2Q)$  
 is time-locally well-posed in $H^m(\TT;\mathbb{C}^n)$ 
 for any integer $m\geqslant 5$. 
\end{theorem}
This is consistent with the observation that the nonlinear
terms generating contributions of the form
$\overline{\p_x^{m+1}U}$ and
$\overline{\p_x^{m+2}U}$ for $U=\p_x^mQ$
are revealed to be essentially harmless
through the diagonalization technique,
so that no additional structural conditions are imposed
on such interactions even in the extended setting.
We can also present sufficient conditions for 
local well-posedness in $H^m(\TT;\CC^n)$ 
for $m\geqslant 4$ 
which correspond to Corollaries~\ref{cor:cor} 
and \ref{cor:cor2}. 
Although no additional difficulty arises in 
proving them, 
we do not state these corresponding results
for simplicity.
\par 
In addition, the inclusion of nonlinear terms involving $G$ suggests a wider range of applicability of our approach. 
In particular, it indicates the possibility of obtaining local well-posedness results in the case 
$M=\diag(a_1,\ldots,a_n)$ with $|a_1|=\cdots=|a_n|\neq 0$. 
To illustrate this point, consider the system \eqref{eq:IVP} with vanishing $G$, 
and assume $M_S=0_n$ and 
$M=\diag(a_1,\ldots,a_n)$ with $a_1=\cdots=a_k=1$ and $a_{k+1}=\cdots=a_n=-1$ for some 
$1\leqslant k<n$. 
Then, by setting $R_j=Q_j$ for $1\leqslant j\leqslant k$ and 
$R_j=\overline{Q_j}$ for $k+1\leqslant j\leqslant n$, 
the system can be transformed into a system for $R=(R^1,\ldots,R^n)$ 
with the principal part given by 
$\p_t R - i I_n \p_x^4 R$,
while the nonlinear terms are transformed into a form
covered by the framework of Theorem~\ref{thm:main2}.
This observation suggests that appropriate structural conditions on the coefficients 
$\omega_{p,q,r}^{k,j}$ can lead to local well-posedness in this setting as well. 
However, carrying out this approach in full generality requires a careful reformulation of the nonlinear terms, 
and remains a nontrivial task. 
\par 
The rest of the paper is organized as follows:  
In Section~\ref{sec:proof}, 
we first explain the idea behind our approach,  
specifying two essential structural conditions.  
We next present a general 
framework to prove local well-posed in $H^m(\TT;\CC^n)$ 
for $m\geqslant 5$ based on the idea. 
In Section~\ref{sec:A2C1C2}, we show that 
the set of the structural conditions (A2) and (C1)--(C2) 
satisfies the assumptions for the above framework 
in Section~\ref{sec:proof} to work, which completes 
the proof of Theorem~\ref{thm:main}.   
In Section~\ref{sec:linear}, 
we study the initial value problem for a class of linear 
dispersive systems 
investigated by Chihara \cite{Chihara2015}. 
This will help to understand the meaning of 
the structural conditions (A1) and (C1)--(C2).  
In Section~\ref{sec:LWP4}, 
we prove Corollary\ref{cor:cor}.  
In Section~\ref{sec:M}, 
we discuss the case $M\neq a\,I_n$ and prove 
Corollary~\ref{cor:cor2}.  
Finally in Section~\ref{sec:G}, 
we show that 
the set of the structural conditions (A2)--(A5) 
and (C1)--(C4) 
satisfies the assumptions for the above framework 
in Section~\ref{sec:proof} to work, which completes 
the proof of Theorem~\ref{thm:main2}.  
\section{Proof under the key conditions}
\label{sec:proof}
In this section, 
we first derive systems satisfied by higher-order 
derivatives of solutions 
and perform a reduction to a form suitable 
for energy estimates, 
where two key conditions ($\widetilde{C1}$)-($\widetilde{C2}$) 
are specified. 
We next present a general framework to prove 
local well-posedness 
of the initial value problem \eqref{eq:IVP} in 
$H^m(\TT;\CC^n)$ for $m\geqslant 5$ 
under ($\widetilde{C1}$)-($\widetilde{C2}$). 
\subsection{Idea of the proof: 
Reduction of the system}
\label{subsec:reduction}
Let $Q=Q(t,x): I_T\times \TT\to \CC^n$ 
be a sufficiently smooth solution to \eqref{eq:IVP}, 
where $I_T\subset \RR$ is a time-interval on which 
$Q$ exists. 
We define $U := \p_x^m Q$ for integer $m\geqslant 5$. 
Applying $\p_x^m$ to \eqref{eq:IVP}, 
we obtain
\begin{align}
&(\p_t - i M \p_x^4 - M_S \p_x^2)U
\nonumber
\\
&= 
P_1(Q)\p_x^2 U
+
P_2(Q)\overline{\p_x^2 U}
+ A_m(Q,\p_xQ)\p_x U 
+ B_m(Q,\p_xQ)\overline{\p_x U} 
+ \text{l.o.t.}, 
\label{eq:U_eq}
\end{align}
where:
\begin{itemize}
\item $P_1(Q)$ and $P_2(Q)$ 
are $n \times n$ 
matrix-valued functions whose entries are 
of order $|Q|^2$,
\item 
$A_m(Q,\p_x Q)$ and $B_m(Q,\p_x Q)$ are 
$n \times n$ matrix-valued functions depending on $m$, 
whose entries are of order $|Q|\,|\p_x Q|$.
\item ``l.o.t.'' denotes lower-order terms, 
which involve spatial derivatives of $Q$ up to order $m$ 
and can be controlled without derivative loss.
\end{itemize}
The first four terms on the right-hand side of \eqref{eq:U_eq} 
involve spatial derivatives of $U$, 
and therefore cannot be directly controlled by the $H^m$ 
norm of $Q$ via the classical energy method.
\par 
We next set 
\[ 
\mathcal{Q}:=(Q\ \ \overline{Q})^{\top},
\quad  
\mathcal{U}:=(U\ \ \overline{U})^{\top}
\] 
and rewrite the system as
\begin{equation}
\label{eq:matrix_system}
\left(\p_t
-\mathcal{M} \p_x^4
- \mathcal{R}(Q) \p_x^2
- \mathcal{R}_m(Q,\p_xQ) \p_x
\right) \mathcal{U}
= \text{l.o.t.},
\end{equation}
where 
\[
\mathcal{M}
=
\begin{pmatrix}
iM & 0_n \\
0_n & \overline{iM}
\end{pmatrix}, 
\quad 
\mathcal{R}(Q) 
=
\begin{pmatrix}
M_S+P_1(Q) & P_2(Q) \\
\overline{P_2(Q)} & \overline{M_S+P_1(Q)}
\end{pmatrix},
\]
and 
\[
\mathcal{R}_m(Q,\p_xQ)
=
\begin{pmatrix}
A_m(Q,\p_xQ) & B_m(Q,\p_xQ) \\
\overline{B_m(Q,\p_xQ)} & \overline{A_m(Q,\p_xQ)}
\end{pmatrix}.
\]
We next decompose the $2n\times 2n$ matrices 
$\mathcal{R}=\mathcal{R}(Q)$ 
and 
$\mathcal{R}_m=\mathcal{R}_m(Q,\p_xQ)$
into Hermitian and skew-Hermitian parts:
\[
\mathcal{R}
= \mathcal{R}_H + \mathcal{R}_S,
\quad
\mathcal{R}_H^* = \mathcal{R}_H, \quad
\mathcal{R}_S^* = - \mathcal{R}_S, 
\]
\[
\mathcal{R}_m
= \mathcal{R}_{m,H} + \mathcal{R}_{m,S},
\quad
\mathcal{R}_{m,H}^* = \mathcal{R}_{m,H}, \quad
\mathcal{R}_{m,S}^* = - \mathcal{R}_{m,S}, 
\]
Furthermore, we split
\[
\mathcal{R}_H
= \mathcal{R}_H^{\mathrm{diag}}
+ \mathcal{R}_H^{\mathrm{off}}, 
\quad 
\mathcal{R}_S
= \mathcal{R}_S^{\mathrm{diag}}
+ \mathcal{R}_S^{\mathrm{off}}, 
\]
\[
\mathcal{R}_{m,H}
= \mathcal{R}_{m,H}^{\mathrm{diag}}
+ \mathcal{R}_{m,H}^{\mathrm{off}}, 
\quad 
\mathcal{R}_{m,S}
= \mathcal{R}_{m,S}^{\mathrm{diag}}
+ \mathcal{R}_{m,S}^{\mathrm{off}},  
\]
where in each case the first term is the block-diagonal part 
and the second is the block-off-diagonal part, 
with respect to the $2\times 2$ block structure, 
in which each block is an $n\times n$ matrix.
\par 
The key idea is to eliminate the problematic terms involving $\p_x^2 \mathcal{U}$ or $\p_x\mathcal{U}$ 
by suitable gauge-type transformations. 
The idea is explained below in two steps.
\par 
First, we introduce a transformed variable 
\begin{equation}
\label{eq:V_def}
\mathcal{V}
= \left(I + \mathcal{L}(Q)\p_x^{-2}
+ \mathcal{L}_{1,m}(Q,\p_xQ)\p_x^{-3}\right)\mathcal{U}, 
\end{equation}
where
\[
I=
I_{2n},
\ 
\mathcal{L}(Q) 
=
\begin{pmatrix}
0_n & M^{-1}\Lambda \\
\overline{M^{-1}\Lambda} & 0_n
\end{pmatrix},
\ 
\mathcal{L}_{1,m}(Q,\p_xQ) 
=
\begin{pmatrix}
0_n & M^{-1}\Lambda_{1,m} \\
\overline{M^{-1}\Lambda_{1,m}} & 0_n
\end{pmatrix},
\]
and $\Lambda=\Lambda(Q)$ and $\Lambda_{1,m}=\Lambda_{1,m}(Q,\p_xQ)$ are
$n\times n$ matrix-valued functions 
defined by 
\begin{align}
\Lambda&=-\dfrac{1}{2}\Phi(Q)iI_n, 
\quad 
\Phi(Q)=
\begin{pmatrix}
I_n & 0_n
\end{pmatrix}
\mathcal{R}_H^{\mathrm{off}}(Q)
\begin{pmatrix}
0_n & I_n
\end{pmatrix}^{\top}, 
\label{eq:Lambda}
\\
\Lambda_{1,m}
&=
-\dfrac{1}{2}\Psi_m(Q,\p_xQ)iI_n, 
\quad 
\Psi_m(Q, \p_xQ)=
\begin{pmatrix}
I_n & 0_n
\end{pmatrix}
\mathcal{K}_{m,S}^{\mathrm{off}}(Q, \p_xQ)
\begin{pmatrix}
0_n & I_n
\end{pmatrix}^{\top}. 
\label{eq:Lambda1m}
\end{align}
In this setting, 
$\mathcal{K}_{m,S}^{\mathrm{off}}(Q, \p_xQ)$ 
denotes the block-off-diagonal part of 
$\mathcal{K}_{m,S}(Q, \p_xQ)$, 
and $\mathcal{K}_{m,S}(Q, \p_xQ)$ 
denotes the skew-Hermitian part of  
$\mathcal{K}_{m}(Q, \p_xQ)$ which is defined by 
\begin{equation}
\mathcal{K}_{m}(Q, \p_xQ)=\mathcal{R}_{m}(Q,\p_xQ)
-\p_x\left(\mathcal{R}(Q)\right). 
\label{eq:K_m}
\end{equation}
Here, $\p_x^{-2}\mathcal{U}$ and 
$\p_x^{-3}\mathcal{U}$ are well-defined, 
since $\mathcal{U}$ is an image of 
$\p_x^m$ for $m\geqslant 3$.   
Under this transformation, the terms acted by
$\mathcal{R}_H^{\mathrm{off}}(Q) \p_x^2$ and 
 $\mathcal{K}_{m,S}^{\mathrm{off}}(Q, \p_xQ)\p_x$
are eliminated so that 
the system \eqref{eq:matrix_system} is reduced to
\begin{align}
\label{eq:V_system}
\p_t \mathcal{V}
&= \mathcal{M} \p_x^4 \mathcal{V}
+ \mathcal{R}_H^{\mathrm{diag}}(Q) 
\p_x^2 \mathcal{V} \notag\\
&\quad
+ \p_x \left\{\mathcal{R}_S(Q) \p_x \mathcal{V}\right\}
+ \mathcal{H}_{m}(Q,\p_xQ)\p_x \mathcal{V}
+\mathcal{K}_{m,S}^{\mathrm{diag}}(Q,\p_xQ)\p_x \mathcal{V}
\nonumber
\\
&\quad 
+O(|\p_tQ||Q||\p_x^{m-2}\mathcal{Q}|)
+O((|\p_t\p_xQ||Q|+|\p_xQ||\p_tQ|)|
\p_x^{m-3}\mathcal{Q}|)
+ \text{l.o.t.}, 
\end{align}
where 
\begin{equation}
\mathcal{H}_{m}(Q,\p_xQ)
=\mathcal{R}_{m,H}(Q,\p_xQ)
-4\mathcal{M}\p_x\left(\mathcal{L}(Q)\right).  
\label{eq:Hm}
\end{equation}
The second term of the right hand side of \eqref{eq:Hm}
is Hermitian. 
Indeed, from the definition of $\mathcal{M}$ and 
$\mathcal{L}(Q)$ with 
\eqref{eq:Lambda}, it follows that 
\begin{align}
\mathcal{M}\p_x\left(\mathcal{L}(Q)\right)
&=
\begin{pmatrix}
0_n & i\p_x(\Lambda(Q)) \\
\overline{i\p_x(\Lambda(Q))} & 0_n
\end{pmatrix}
=\dfrac{1}{2}
\begin{pmatrix}
0_n & \p_x(\Phi(Q)) \\
\overline{\p_x(\Phi(Q))} & 0_n
\end{pmatrix}
\nonumber
\\
&
=\dfrac{1}{2}\p_x(
\mathcal{R}_{H}^{\mathrm{off}}(Q)
). 
\nonumber
\end{align}
Since the right-hand side is expressed as a derivative 
of $\mathcal{R}_{H}^{\mathrm{off}}(Q)$, 
which is Hermitian, 
it follows that it is also Hermitian.
Therefore $\mathcal{H}_{m}(Q,\p_xQ)$ is also Hermitian, 
which shows the term $\mathcal{H}_{m}(Q,\p_xQ)\p_x \mathcal{V}$ behaves as a first-order term 
of a symmetric hyperbolic system 
in that it causes no loss of derivatives 
via integration by parts. 
By the skew-Hermitian property of $\mathcal{R}_S(Q)$, 
the term $\p_x \left\{\mathcal{R}_S(Q) \p_x \mathcal{V}\right\}$ is also harmless 
via integration by parts. 
In contrast, under the assumption $m\geqslant 5$, 
the genuinely problematic terms are those acted by  
$\mathcal{R}_H^{\mathrm{diag}}(Q)\p_x^2$
or $\mathcal{K}_{m,S}^{\mathrm{diag}}(Q,\p_xQ)\p_x$. 
\par 
Second, we introduce and impose the two essential conditions:  
\begin{align}
	\tag{$\widetilde{C1}$} 
	\mathcal{R}_H^{\mathrm{diag}}(Q) 
	&= 0,
\\
	\tag{$\widetilde{C2}$}
\begin{pmatrix}
	I_n & 0_n
\end{pmatrix}
\mathcal{K}_{m,S}^{\mathrm{diag}}(Q,\p_xQ)
\begin{pmatrix}
	I_n & 0_n
\end{pmatrix}^{\top}
&=
\p_x\left(
D_m(Q)
\right), 
\end{align}
where $D_m(Q)$ is an $n\times n$ matrix-valued function
defined on $I_T\times\TT$,
which depends on $m$ and is of order $|Q|^2$.
Under the assumptions ($\widetilde{C1}$)-($\widetilde{C2}$), 
the term acted by 
$\mathcal{R}_H^{\mathrm{diag}}(Q)\p_x^2$ vanishes and 
the term acted by $\mathcal{K}_{m,S}^{\mathrm{diag}}(Q,\p_xQ)\p_x$ 
is also eliminated by introducing a new transformed variable:  
\begin{equation}
	\label{eq:W_def}
	\mathcal{W}
	= \left(I + \mathcal{L}(Q)\p_x^{-2}
	+ \mathcal{L}_{1,m}(Q,\p_xQ)\p_x^{-3}
	- \mathcal{L}_{2,m}(Q)\p_x^{-2}
	\right)\mathcal{U}. 
\end{equation}
Here, $I$, $\mathcal{L}(Q)$, and $\mathcal{L}_{1,m}(Q,\p_xQ)$ 
are the same as those introduced above, 
\[
\mathcal{L}_{2,m}(Q) 
=
\begin{pmatrix}
	M^{-1}\Lambda_{2,m}  & 0_n\\
	0_n & \overline{M^{-1}\Lambda_{2,m}}
\end{pmatrix},
\]
and 
\begin{equation}
\Lambda_{2,m}=\Lambda_{2,m}(Q)
=\dfrac{1}{4}D_m(Q)iI_n. 
\label{eq:Lambda2m}
\end{equation}
Indeed, a simple computation using \eqref{eq:W_def} and 
$\mathcal{U}=\mathcal{W}+O(|Q|^2)\p_x^{m-2}\mathcal{Q}
+O(|Q||\p_xQ|)\p_x^{m-3}\mathcal{Q}$
yields  
\begin{align}
\p_t \mathcal{W}
&=
\mathcal{M} \p_x^4\mathcal{W}
+\mathcal{R}(Q) \p_x^2\mathcal{W}
+\mathcal{R}_m(Q,\p_xQ) \p_x\mathcal{W}
\nonumber
\\
&\quad 
-\left[
\mathcal{M}\p_x^4, \mathcal{L}(Q)\p_x^{-2}
\right]\mathcal{U}
-\left[
\mathcal{M}\p_x^4, \mathcal{L}_{1,m}(Q,\p_xQ)\p_x^{-3}
\right]\mathcal{U}
+\left[
\mathcal{M}\p_x^4, \mathcal{L}_{2,m}(Q)\p_x^{-2}
\right]\mathcal{U}
\nonumber
\\
&\quad 
+\p_t(\mathcal{L}(Q))\p_x^{m-2}\mathcal{Q}
+\p_t(\mathcal{L}_{1,m}(Q,\p_xQ))\p_x^{m-3}\mathcal{Q}
+\p_t(\mathcal{L}_{2,m}(Q))\p_x^{m-2}\mathcal{Q}
+\text{l.o.t.}. 
\label{eq:W12}
\end{align}
Here 
\[
\left[
\mathcal{M}\p_x^4, \mathcal{L}(Q)\p_x^{-2}
\right]\mathcal{U}
=
\left[
\mathcal{M}, \mathcal{L}(Q)
\right]\p_x^2\mathcal{U}
+
4\mathcal{M}\p_x\left(
\mathcal{L}(Q)
\right)
\p_x\mathcal{U}
+\text{l.o.t.}. 
\]
Moreover, using $[M, \Lambda(Q)]=0_n$ which is ensured by 
the assumption $M=a I_n$, and using
\eqref{eq:Lambda} and the definition of $\mathcal{R}(Q)$, 
we deduce 
\begin{align}
\left[
\mathcal{M}, \mathcal{L}(Q)
\right]
&=
\begin{pmatrix}
	0_n & 2i\Lambda \\
	\overline{2i\Lambda} & 0_n
\end{pmatrix}
=
\begin{pmatrix}
	0_n & \Phi(Q) \\
	\overline{\Phi(Q)} & 0_n
\end{pmatrix}
=
\mathcal{R}_H^{\mathrm{off}}(Q). 
\nonumber
\end{align}
Hence we have 
\begin{equation}
\left[
\mathcal{M}\p_x^4, \mathcal{L}(Q)\p_x^{-2}
\right]\mathcal{U}
=
\mathcal{R}_H^{\mathrm{off}}(Q)
\p_x^2\mathcal{U}
+
4\mathcal{M}\p_x\left(
\mathcal{L}(Q)
\right)
\p_x\mathcal{U}
+\text{l.o.t.}. 
\label{eq:Clambda}
\end{equation}
In the same way as above, 
using $[M, \Lambda_{1,m}(Q,\p_xQ)]=0_n$ 
ensured by $M=a I_n$, and using
\eqref{eq:Lambda1m} and the definitions of 
$\mathcal{R}(Q)$ and $\mathcal{R}_{1,m}(Q,\p_xQ)$, 
we deduce 
\begin{align}
	\left[
	\mathcal{M}, \mathcal{L}_{1,m}(Q,\p_xQ)
	\right]
	&=
	\begin{pmatrix}
		0_n & 2i\Lambda_{1,m} \\
		\overline{2i\Lambda_{1,m}} & 0_n
	\end{pmatrix}
	=
	\begin{pmatrix}
		0_n & \Psi_m \\
		\overline{\Psi_m} & 0_n
	\end{pmatrix}
	=
	\mathcal{K}_{m,S}^{\mathrm{off}}(Q,\p_xQ), 
	\nonumber
\end{align}
and hence 
\begin{align}
\left[
\mathcal{M}\p_x^4, \mathcal{L}_{1,m}(Q,\p_xQ)\p_x^{-3}
\right]\mathcal{U}
&=
\left[
\mathcal{M}, \mathcal{L}_{1,m}(Q,\p_xQ)
\right]\p_x\mathcal{U}
+\text{l.o.t.}.
\nonumber
\\
&=
\mathcal{K}_{m,S}^{\mathrm{off}}(Q,\p_xQ)
\p_x\mathcal{U}
+\text{l.o.t.}. 
\label{eq:Clambda1m}
\end{align}
In contrast, we see 
$\left[
\mathcal{M}, \mathcal{L}_{2,m}(Q)
\right]=0_{2n}$ 
follows from $[M, \Lambda_{2,m}(Q)]=0_n$ 
which is ensured by $M=a I_n$.  
This is because 
$\mathcal{L}_{2,m}(Q)$ consists only of components in 
the block-diagonal part. 
Noting this and using \eqref{eq:Lambda2m} and 
the condition ($\widetilde{C2}$), we obtain 
\begin{align}
	\left[
	\mathcal{M}\p_x^4, \mathcal{L}_{2,m}(Q)\p_x^{-2}
	\right]\mathcal{U}
	&=
	4\mathcal{M}\p_x\left(
	\mathcal{L}_{2,m}(Q)
	\right)
	\p_x\mathcal{U}
	+\text{l.o.t.}
	\nonumber
	\\
	&=
	4\begin{pmatrix}
		i\p_x(\Lambda_{2,m})  & 0_n\\
		0_n & \overline{i\p_x(\Lambda_{2,m})}
	\end{pmatrix}
	\p_x\mathcal{U}
	+\text{l.o.t.}
	\nonumber
	\\&=
	-\begin{pmatrix}
		\p_x(D_{m}(Q))  & 0_n\\
		0_n & \overline{\p_x(D_{m}(Q))}
	\end{pmatrix}
	\p_x\mathcal{U}
	+\text{l.o.t.}
	\nonumber
	\\&=
	-\mathcal{K}_{m,S}^{\mathrm{diag}}(Q,\p_xQ)\p_x\mathcal{U}
	+\text{l.o.t.}.
	\label{eq:Clambda2m}
\end{align}
In addition, $\p_t(\Lambda(Q))$ and $\p_t(\Lambda_{2,m}(Q))$ 
are of order $|\p_tQ||Q|$ and $\p_t(\Lambda_{1,m}(Q,\p_xQ))$ is of order 
$|\p_t\p_xQ||Q|+|\p_xQ||\p_tQ|$. 
Noting them and substituting \eqref{eq:Clambda}-\eqref{eq:Clambda2m} into 
\eqref{eq:W12}, 
and then using 
$\mathcal{U}=\mathcal{W}+O(|Q|^2)\p_x^{m-2}\mathcal{Q}
+O(|Q||\p_xQ|)\p_x^{m-3}\mathcal{Q}$, 
we deduce 
\begin{align}
	\p_t \mathcal{W}
	&=
	\mathcal{M} \p_x^4\mathcal{W}
	+\mathcal{R}(Q) \p_x^2\mathcal{W}
	+\mathcal{R}_m(Q,\p_xQ) \p_x\mathcal{W}
	\nonumber
	\\
	&\quad 
	-\mathcal{R}_H^{\mathrm{off}}(Q)
	\p_x^2\mathcal{W}
	-
	4\mathcal{M}\p_x\left(
	\mathcal{L}(Q)
	\right)
	\p_x\mathcal{W}
	-\mathcal{K}_{m,S}^{\mathrm{off}}(Q,\p_xQ)
	\p_x\mathcal{W}
	-\mathcal{K}_{m,S}^{\mathrm{diag}}(Q,\p_xQ)\p_x\mathcal{W}
	\nonumber
	\\
	&\quad 
	+O(|\p_tQ||Q||\p_x^{m-2}\mathcal{Q}|)
	+O((|\p_t\p_xQ||Q|+|\p_xQ||\p_tQ|)|
	\p_x^{m-3}\mathcal{Q}|)
	+ \text{l.o.t.}. 
	\nonumber
\end{align}
Here, by the condition ($\widetilde{C1}$), 
we see  
$\mathcal{R}=\mathcal{R}_H+\mathcal{R}_S
=\mathcal{R}_H^{\mathrm{off}}+\mathcal{R}_S$ 
and hence 
\[
\mathcal{R}(Q) \p_x^2\mathcal{W}
	-\mathcal{R}_H^{\mathrm{off}}(Q)
\p_x^2\mathcal{W}
=
\mathcal{R}_S(Q) \p_x^2\mathcal{W}
=
\p_x\left\{
\mathcal{R}_S(Q) \p_x\mathcal{W}
\right\}
-\p_x(\mathcal{R}_S(Q))\p_x\mathcal{W}.
\]
Using this, 
$\mathcal{R}_m=\mathcal{R}_{m,H}+\mathcal{R}_{m,S}$, 
and 
$\mathcal{K}_{m,S}=\mathcal{K}_{m,S}^{\mathrm{diag}}+\mathcal{K}_{m,S}^{\mathrm{off}}$, 
we have 
\begin{align}
	\p_t \mathcal{W}
	&=
	\mathcal{M} \p_x^4\mathcal{W}
	+\p_x\left\{
	\mathcal{R}_S(Q) \p_x\mathcal{W}
	\right\}
	+\left\{
	\mathcal{R}_{m,H}(Q,\p_xQ) -
	4\mathcal{M}\p_x\left(
	\mathcal{L}(Q)
	\right)
	\right\}
	\p_x\mathcal{W}
	\nonumber
	\\
	&\quad
	+\left\{
	\mathcal{R}_{m,S}(Q,\p_xQ) 
		-\p_x(\mathcal{R}_S(Q))
	-\mathcal{K}_{m,S}(Q,\p_xQ)
	\right\}
	\p_x\mathcal{W}
	\nonumber
	\\
	&\quad 
	+O(|\p_tQ||Q||\p_x^{m-2}\mathcal{Q}|)
	+O((|\p_t\p_xQ||Q|+|\p_xQ||\p_tQ|)|
	\p_x^{m-3}\mathcal{Q}|)
	+ \text{l.o.t.}. 
	\nonumber
\end{align}
By the setting \eqref{eq:K_m} and \eqref{eq:Hm}, 
both the third and fourth terms of right hand side are 
canceled out, and thus  
we conclude that the system \eqref{eq:matrix_system} 
is reduced to 
\begin{align}
	\label{eq:W_system}
	\p_t \mathcal{W}
	&= \mathcal{M} \p_x^4 \mathcal{W}
	+ \p_x \left\{\mathcal{R}_S(Q) \p_x \mathcal{W}\right\}
	+ \mathcal{H}_{m}(Q,\p_xQ)\p_x \mathcal{W}
	\nonumber
	\\
	&\quad 
	+O(|\p_tQ||Q||\p_x^{m-2}\mathcal{Q}|)
	+O((|\p_t\p_xQ||Q|+|\p_xQ||\p_tQ|)|
	\p_x^{m-3}\mathcal{Q}|)
	+ \text{l.o.t.}, 
\end{align}
which no longer contains terms leading to derivative loss 
in the energy estimates. 
\par 
It is also convenient to reduce 
\eqref{eq:W_system} back to the original 
$n$-component system:   
Before doing this, let us observe the following 
$2n\times 2n$ matrix 
\begin{equation}
\mathcal{C}=
\begin{pmatrix}
	A & B \\
	\overline{B} & \overline{A}
\end{pmatrix}, 
\label{eq:MatC}
\end{equation}
where $A,B$ are $n\times n$ matrices. 
It is an elementary property that the following equivalences hold, respectively: 
\begin{align}
\mathcal{C}^{*}=\mathcal{C} 
&\iff 
A=A^{*} \ \text{and}\ B=B^{\top}, 
\nonumber
\\
\mathcal{C}^{*}=-\mathcal{C} 
&\iff 
A=-A^{*} \ \text{and}\ B=-B^{\top}.
\nonumber
\end{align}
By definition, both $\mathcal{R}_S(Q)$ with $\mathcal{R}_S(Q)^{*}=-\mathcal{R}_S(Q)$
and $\mathcal{H}_{m}(Q,\p_xQ)$ with $\mathcal{H}_{m}(Q,\p_xQ)^{*}=\mathcal{H}_{m}(Q,\p_xQ)$
are expressed 
as the form \eqref{eq:MatC}. Thus, we can express 
\begin{align}
	\mathcal{R}_S(Q)
	&=
	\begin{pmatrix}
		J(Q) & K(Q) \\
		\overline{K(Q)} & \overline{J(Q)}
	\end{pmatrix}, 
	\quad 
\mathcal{H}_{m}(Q,\p_xQ)
=
\begin{pmatrix}
	H_m(Q,\p_xQ) & S_m(Q,\p_xQ) \\
	\overline{S_m(Q,\p_xQ)} & \overline{H_m(Q,\p_xQ)}
\end{pmatrix}, 
\nonumber
\end{align}
where $J(Q)=-J(Q)^{*}$, $K(Q)=-K(Q)^{\top}$, $H_m(Q,\p_xQ)=H_m(Q,\p_xQ)^{*}$, and
$S_m(Q,\p_xQ)=S_m(Q,\p_xQ)^{\top}$. 
Moreover, by the definition \eqref{eq:W_def}, we can express 
$\mathcal{W}=(W\  \overline{W})^{\top}$ for $\CC^n$-valued function $W$. 
Therefore, by taking out the first $n$-components of \eqref{eq:W_system}, 
we derive 
\begin{align}
\p_tW&=ia\,\p_x^4W+\p_x\{
J(Q)\p_xW+K(Q)\overline{\p_xW}
\}
\nonumber
\\
&\quad 
+H_m(Q,\p_xQ)\p_xW+S_m(Q,\p_xQ)\overline{\p_xW}
\nonumber
\\
&\quad 
+O(|\p_tQ||Q||\p_x^{m-2}Q|)
+O((|\p_t\p_xQ||Q|+|\p_xQ||\p_tQ|)|
\p_x^{m-3}Q|)
+ \text{l.o.t.}. 
\label{eq:W_system2}
\end{align}
\par 
This reduction is the fundamental step in our analysis 
and will allow us 
to construct modified energies to complete the proof of 
Theorem~\ref{thm:main}.
\subsection{Proof of local well-posedness in $H^m$ 
for $m\geqslant 5$
based on ($\widetilde{C1}$)-($\widetilde{C2}$)} 
In this part, we present a framework
to prove local well-posedness 
of the initial value problem \eqref{eq:IVP} 
in $H^m(\TT;\CC^n)$ 
for any integer $m\geqslant 5$
by imposing 
($\widetilde{C1}$)--($\widetilde{C2}$) in certain situations. 
Apart from the difference in the form of the gauge-type transformations,
the overall strategy of the proof is essentially the same as that presented in detail in \cite{Onodera2025}.
Therefore, we omit some of the details and 
only sketch the argument. 
Furthermore, Theorem~\ref{thm:main} 
follows from the proof given here,
since its structural conditions satisfy ($\widetilde{C1}$)--($\widetilde{C2}$)
for all $m \ge 5$ and also for $m=2$, 
which will be verified in Section~\ref{sec:A2C1C2}. 
\par 
Let $m\geqslant 5$ be an integer, and let
$Q_0\in H^m(\RR;\mathbb{C}^n)$. 
From the time-reversibility, 
it suffices to solve \eqref{eq:IVP} 
in positive time-direction. 
\vspace{0.5em}
\par 
\noindent {\bf Uniform $H^m$-estimate for approximated solutions.}
Let $\left\{Q_{0}^{\ep}\right\}_{\ep\in (0,1)}$ be  
the so-called Bona-Smith approximation of $Q_0$ which 
is constructed to satisfy  
$Q_0^{\ep}\in H^{\infty}(\TT;\mathbb{C}^n)$, 
$Q_0^{\ep}\to Q_0$ in $H^m(\TT;\mathbb{C}^n)$ 
as $\ep \downarrow 0$, 
and  
\begin{align}
	&\|Q_0^{\ep}\|_{H^m}
	\leqslant 
	\|Q_0\|_{H^m}, 
	\label{eq:bs1}
	\\
	&\|Q_0^{\ep}\|_{H^{m+\ell}}
	\leqslant 
	C\ep^{-\ell}
	\|Q_0\|_{H^m}
	\quad 
	(\ell=0,1,2,\ldots),
	\label{eq:6092}
	\\
	&\|Q_0^{\ep}-Q_0\|_{H^{m-\ell}}
	\leqslant 
	C\ep^{\ell}
	\|Q_0\|_{H^m}
	\quad 
	(\ell=0,1,2,\ldots), 
	\label{eq:6093}
\end{align} 
where $C>0$ is a  
constant which depends on $m,\ell,\phi$, 
but not on $\ep$. 
Taking $Q_0^{\ep}$ for $\ep\in (0,1)$ as initial data, 
we consider the following fourth-order parabolic 
regularized problem 
\begin{alignat}{2}
	\left(
	\p_t+\ep^5 \p_x^4-iM\p_x^4-M_S\p_x^2
	\right)
	Q
	&=F(Q, \p_xQ, \p_x^2Q), 
	\label{eq:bpde}
	\\
	Q(0,x)
	&=
	Q_{0}^{\ep}(x) 
	\label{eq:bdata}
\end{alignat}
for $Q=(Q_1,\ldots,Q_n)^{\top}:[0,\infty)\times \TT\to \mathbb{C}^n$. 
By standard parabolic theory, there exist
$T_{\ep}=T(\ep, \|Q_0^{\ep}\|_{H^m})>0$ 
and a unique smooth solution 
$Q^{\ep}=(Q^{\ep}_1,\ldots,Q^{\ep}_n)^{\top}
\in C([0,T_{\ep}];H^{\infty}(\TT;\mathbb{C}^n))$.  
\par 
Based on the reduction performed in 
Section~\ref{subsec:reduction}, 
we introduce a $\CC^n$-valued function 
$W_m^{\ep}=W_m^{\ep}(t,x)$ 
defined by 
\begin{align}
W_m^{\ep} 
&:= \p_x^m Q^{\ep} 
+M^{-1}\Lambda(Q^{\ep})\overline{\p_x^{m-2}Q^{\ep}} 
+M^{-1}\Lambda_{1,m}(Q^{\ep},\p_xQ^{\ep})\overline{\p_x^{m-3}Q^{\ep}} 
\nonumber
\\
&\qquad 
-M^{-1}\Lambda_{2,m}(Q^{\ep})\p_x^{m-2}Q^{\ep}, 
\label{eq:Wvar}
\end{align}
where $\Lambda$, $\Lambda_{1,m}$ and  $\Lambda_{2,m}$ are $n\times n$ matrix-valued functions on 
$[0,T_{\ep}]\times \TT$
determined by \eqref{eq:Lambda}, \eqref{eq:Lambda1m}, and \eqref{eq:Lambda2m} 
respectively with $Q$ replaced by $Q^{\ep}$.
We note that \eqref{eq:Wvar} corresponds to the first $n$-components 
of the $2n$-component variable defined in \eqref{eq:W_def}.
Furthermore, we introduce a modified energy 
$\mathcal{E}_m(Q^{\ep})=\mathcal{E}_m(Q^{\ep}(t))$ defined by 
\begin{equation}
	\mathcal{E}_m(Q^{\ep}(t))
	:= \sqrt{\|W_m^{\ep}(t) \|_{L^2}^2 + \|Q^{\ep}(t) \|_{H^{m-1}}^2}. 
		\label{eq:Em_def}
\end{equation}
By the Sobolev embedding, there exist  
some constants $C_1,C_2>0$ which are 
independent of $\ep\in (0,1)$ such that   
\begin{equation}
	\frac{\|Q^{\ep}(t)\|_{H^m}^2}{C_1\,(1+\|Q^{\ep}(t)\|_{H^2}^2)}
	\leqslant 
	\mathcal{E}_m(Q^{\ep}(t))^2
	\leqslant 
	C_2\,
	(1+\|Q^{\ep}(t)\|_{H^2}^2)
	\|Q^{\ep}(t)\|_{H^m}^2
	\label{eq:eqi9231}
\end{equation}
on $[0,T_{\ep}]$.  
Furthermore, we set 
\begin{equation}
	T^{\star}_{\ep}
	=
	\sup
	\left\{
	T>0 \ | \ 
	\mathcal{E}_5(Q^{\ep}(t))\leqslant 2\,\mathcal{E}_5(Q^{\ep}(0))
	\ \ 
	\text{for all}
	\ \
	t\in[0,T]
	\right\}. 
	\label{eq:cuttime}
\end{equation}
By \eqref{eq:Em_def}-\eqref{eq:cuttime} 
and $\|Q_0^{\ep}\|_{H^5}\leqslant \|Q_0\|_{H^5}$ 
which follows from \eqref{eq:bs1} for $m=5$, 
there exist  
$C_k=C_k(\|Q_0\|_{H^5})>0$ for $k=3,4,5$ depending on 
$\|Q_0\|_{H^5}$ but not on $\ep\in (0,1)$ such that  
\begin{align}
	\sup_{t\in [0,T_{\ep}^{\star}]}\|Q^{\ep}(t)\|_{H^5}^2
	&\leqslant 
	C_3(\|Q_0\|_{H^5})    
	\label{eq:H3Emn}
\end{align}
and 
\begin{align}
	&\frac{\|Q^{\ep}(t)\|_{H^m}^2}{C_4(\|Q_0\|_{H^5})}
	\leqslant 
	\mathcal{E}_m(Q^{\ep}(t))^2
	\leqslant 
	C_5(\|Q_0\|_{H^5})\|Q^{\ep}(t)\|_{H^m}^2
	\ 
	\text{on $[0,T_{\ep}^{\star}]$}.
	\label{eq:H3Emm}
\end{align} 
This shows  $\mathcal{E}_m(Q^{\ep}(t))$
is equivalent to the standard $H^m$ norm of $Q^{\ep}(t)$ 
on $[0,T_{\ep}^{\star}]$.
\par 
Now we impose ($\widetilde{C1}$)-($\widetilde{C2}$) with $m\geqslant 5$
for $Q$ replaced by $Q^{\ep}$. 
Repeating the derivation of \eqref{eq:W_system2}
and taking into account the added parabolic term
$\ep^5\p_x^4Q$ in \eqref{eq:bpde},
we obtain an $n$-component system
satisfied by $W_m^{\ep}$ of the form 
\begin{align}
	\p_t W_m^{\ep}
	&= (-\ep^5I_n+ i\,M )\p_x^4 W_m^{\ep}
	+ \p_x\left\{J(Q^{\ep})\p_x W_m^{\ep}
	+K(Q^{\ep})\overline{\p_x W_m^{\ep}}
	\right\}
	\nonumber
	\\
	&\quad 
	+ H_m(Q^{\ep},\p_xQ^{\ep})\p_x W_m^{\ep}
	+S_m(Q^{\ep},\p_xQ^{\ep}) \overline{\p_x W_m^{\ep}}
	\nonumber
	\\
	&\quad 
	+O(|\p_tQ^{\ep}||Q^{\ep}||\p_x^{m-2}Q^{\ep}|)
	+O((|\p_t\p_xQ^{\ep}||Q^{\ep}|+|\p_xQ^{\ep}|
	|\p_tQ^{\ep}|)|
	\p_x^{m-3}Q^{\ep}|)
	\nonumber
	\\&\quad 
	+\ep^5\left\{
	O(|Q^{\ep}||\p_xQ^{\ep}||\p_x^{m+1}Q^{\ep}|)+r_1
	\right\}
	 +r_2, 
	 \label{eq:sys_Wep}
\end{align}
where $J(Q^{\ep})$, $K(Q^{\ep})$, $H_m(Q^{\ep},\p_xQ^{\ep})$, and $S_m(Q^{\ep},\p_xQ^{\ep})$ 
are $n\times n$ matrix-valued functions satisfying the following: 
\begin{itemize}
	\item $J(Q^{\ep})$ is  
	skew-Hermitian, and the entries 
	are of order $1+|Q^{\ep}|^2$. 
	\item $K(Q^{\ep})$ is 
	skew-symmetric, and the entries 
	are of order $|Q^{\ep}|^2$. 
	\item $H_m(Q^{\ep},\p_xQ^{\ep})$ is Hermitian, and the entries are 
	of order $|Q^{\ep}||\p_xQ^{\ep}|$. 
	\item 
	$S_m(Q^{\ep},\p_xQ^{\ep})$ is symmetric, and the entries 
	are of order $|Q^{\ep}||\p_xQ^{\ep}|$. 
\end{itemize}
No loss of derivatives occurs from lower-order terms acted by them in \eqref{eq:sys_Wep}, 
by integration by parts. 
Moreover, 
\begin{itemize}
	\item $r_1$ and $r_2$ consist of ``l.o.t.'' which satisfy 
	$$\|r_j(t)\|_{L^2}\leqslant
	C(m,\|Q^{\ep}(t)\|_{H^5})
	\|Q^{\ep}(t)\|_{H^m} \quad (j=1,2).
	$$ 
	\item 
	The second-to-last term on the right-hand side,
	which contains the factor $\ep^5$,
	consists of the first $n$ components of
	\[
	\left[\ep^5I\p_x^4,
	\mathcal{L}(Q)\p_x^{-2}
	+\mathcal{L}_{1,m}(Q)\p_x^{-3}
	-\mathcal{L}_{2,m}(Q)\p_x^{-2}
	\right]
	(\p_x^mQ^{\ep}\ \ \overline{\p_x^mQ^{\ep}})^{\top}
	\]
	and can be absorbed into the contribution of
	$-\ep^5\p_x^4W_m^\ep$
	in the energy estimate.
	\item The assumption $m\geqslant 5$ ensures that 
	the $L^2$-norm in $x$ of the third and the fourth to last terms on the right-hand side including $\p_tQ^{\ep}$ or $\p_t\p_xQ^{\ep}$ are bounded by $C(m,\|Q^{\ep}(t)\|_{H^5})
		\|Q^{\ep}(t)\|_{H^m}$. 
\end{itemize}
By \eqref{eq:sys_Wep} and \eqref{eq:eqi9231},
the classical energy estimates based on integration by parts
and the Sobolev embedding yield  
\begin{align}
&
\dfrac{d}{dt}
\|W_m^{\ep}(t)\|_{L^2}^2
+\dfrac{\ep^5}{2}\|\p_x^2W_m^{\ep}(t)\|_{L^2}^2
\leqslant 
C(m,\|Q^{\ep}(t)\|_{H^5})\mathcal{E}_m(Q^{\ep}(t))^2. 
\label{eq:9248}
\end{align}
Permitting loss of one-derivative, it is easy to see 
the time-derivative of $\|Q^{\ep}(t)\|_{H^{m-1}}^2$ 
is bounded by 
$C(m,\|Q^{\ep}(t)\|_{H^5})
\|Q^{\ep}(t)\|_{H^m}^2$. 
Combining them and using  \eqref{eq:H3Emn}-\eqref{eq:H3Emm}, 
we derive 
\begin{align}
&
\dfrac{d}{dt}
\mathcal{E}_m(Q^{\ep}(t))^2
\leqslant 
C(m,\|Q_0\|_{H^5})\mathcal{E}_m(Q^{\ep}(t))^2
\quad
\text{on $[0,T_{\ep}^{\star}]$}. 
\nonumber
\end{align}
Here, $C(m,\|Q_0\|_{H^5})>0$ denotes 
a constant depending on $m$ and $\|Q_0\|_{H^5}$ 
but not on $\ep\in (0,1)$. 
Furthermore, the Gronwall inequality shows  
\begin{align}
\mathcal{E}_m(Q^{\ep}(t))^2
&\leqslant 
\mathcal{E}_m(Q^{\ep}(0))^2
\exp(C(m,\|Q_0\|_{H^5})T_{\ep}^{\star})
\quad
\text{on $[0,T_{\ep}^{\star}]$}.
\nonumber
\end{align}
Combining the above estimate with the definition of
$T_{\ep}^{\star}$, \eqref{eq:H3Emm},
\eqref{eq:eqi9231} evaluated at $t=0$,
and \eqref{eq:bs1},
we see there exists $T=T(\|Q_0\|_{H^5})>0$ 
satisfying $T\leqslant T_{\ep}^{\star}$ uniformly in 
$\ep\in (0,1)$ such that   
\begin{align}
\sup_{t\in [0,T]}
\|Q^{\ep}(t)\|_{H^m}^2
&\leqslant 
C(T, \|Q_0\|_{H^5}, m)
\|Q_0^{\ep}\|_{H^m}^2
\leqslant 
C(T, \|Q_0\|_{H^5}, m)
\|Q_0\|_{H^m}^2.
\label{eq:H153}
\end{align}
This concludes that $\left\{Q^{\ep}\right\}_{\ep\in (0,1)}$ 
is bounded in $L^{\infty}(0,T; H^m(\TT;\CC^n))$. 
\vspace{0.5em}
\par 
\noindent 
{\bf Estimates for the difference of $Q^{\mu}$ and $Q^{\nu}$.}
Let $Q^{\mu}$ and $Q^{\nu}$ denote the solutions to 
\eqref{eq:bpde}-\eqref{eq:bdata} 
with $(\ep,Q_0^{\ep})$ replaced by $(\mu,Q_0^{\mu})$ and that by $(\nu,Q_0^{\nu})$ 
respectively.
The estimate \eqref{eq:H153} for 
$(Q^{\ep}, Q_0^{\ep})$
also holds for 
$(Q^{\mu}, Q_0^{\mu})$ and 
$(Q^{\nu}, Q_0^{\nu})$.
\par 
We impose ($\widetilde{C1}$)-($\widetilde{C2}$) with $m\geqslant 5$ and with $m=2$
for $Q$ replaced by $Q^{\mu}$.  
Then we can show the following: 
\begin{proposition}
\label{proposition:cauchy}
There exists a constant $C=C(T,\|Q_0\|_{H^m},m)>1$ such that 
for all $\mu$ and $\nu$ satisfying 
$0<\mu\leqslant \nu<1$, 
\begin{align}
\|Q^{\mu}-Q^{\nu}\|_{C([0,T];H^2)}
&\leqslant 
C
(\nu^{m-2}+\nu^4),
\label{eq:1cauchy}
\\
\|Q^{\mu}-Q^{\nu}\|_{C([0,T];H^m)}
&\leqslant 
C\left(
\nu^{m-4}+\nu
+\|Q_0^{\mu}-Q_0^{\nu}\|_{H^m}
\right).
\label{eq:mcauchy}
\end{align}  
\end{proposition}
\begin{proof}[Proof of \eqref{eq:1cauchy} in Proposition~\ref{proposition:cauchy}]
For $\mu, \nu$ satisfying 
$0<\mu \leqslant \nu<1$, 
we set $W:=Q^{\mu}-Q^{\nu}$.  
Since $Q^{\mu}$ and $Q^{\nu}$ satisfy \eqref{eq:bpde} 
with $\ep=\mu$ and that with $\ep=\nu$
respectively, 
\begin{align}
&\left(
 \p_t+\mu^5\p_x^4-iM\, \p_x^4-M_S\p_x^2
 \right)\p_x^2W
\nonumber
\\
&=
\p_x^2\left\{
F(Q^{\mu},\p_xQ^{\mu},\p_x^2Q^{\mu})\right\}
 -\p_x^2\left\{F(Q^{\nu},\p_xQ^{\nu},\p_x^2Q^{\nu})\right\}
 +(\nu^5-\mu^5)\p_x^4(\p_x^2Q^{\nu}). 
 \nonumber
\end{align}
A simple computation shows 
\begin{align}
&\left(
 \p_t+\mu^5\p_x^4-iM\, \p_x^4-M_S\p_x^2
 \right)\p_x^2W
\nonumber
\\
&=
(\nu^5-\mu^5)\p_x^6Q^{\nu}
+
 P_1(Q^{\mu})\p_x^4W
 +
 P_2(Q^{\mu})\overline{\p_x^4W} 
 \nonumber
 \\
 &\quad  
 +
 A_2(Q^{\mu},\p_xQ^{\mu})\p_x^3W
  +
  B_2(Q^{\mu},\p_xQ^{\mu})\overline{\p_x^3W}
 +r_1, 
 \label{eq:di6121}
 \end{align}
where $P_1(Q^{\mu})$, $P_2(Q^{\mu})$, 
$A_2(Q^{\mu},\p_xQ^{\mu})$ and $B_2(Q^{\mu},\p_xQ^{\mu})$ are the same 
$n\times n$ matrix-valued functions 
appearing in   \eqref{eq:U_eq}  
with $(Q,m)$ replaced by $(Q^{\mu},2)$, 
and 
$$
\|r_1(t)\|_{L^2}
\leqslant 
C(\|Q^{\mu}(t)\|_{H^3}, \|Q^{\nu}(t)\|_{H^4}) \|W(t)\|_{H^2} 
\quad \text{for any $t\in [0,T]$.}
$$
\par 
We introduce 
$\mathcal{E}_2^{\mu, \nu}(W)
=\mathcal{E}_2^{\mu, \nu}(W(t))$ 
defined by 
\begin{align}
\mathcal{E}_2^{\mu, \nu}(W(t))
&=
\frac{1}{2}
\|
\p_x^2W(t)
+M^{-1}\Lambda(Q^{\mu})\overline{W}(t)
-M^{-1}\Lambda_{2,2}(Q^{\mu})W(t)
\|_{L^2}^2
\nonumber
\\
&\quad 
-\operatorname{Re}
\langle
\p_xW(t), M^{-1}\Lambda_{1,2}(Q^{\mu},\p_xQ^{\mu})
\overline{W}(t)
\rangle
+A\, \|W(t)\|_{H^1}^2, 
\label{eq:5313}
\end{align}
where 
$\Lambda_1(Q^{\mu})$, 
$\Lambda_{1,2}(Q^{\mu},\p_xQ^{\mu})$ and 
$\Lambda_{2,2}(Q^{\mu})$ are determined by 
\eqref{eq:Lambda}, \eqref{eq:Lambda1m},   
and \eqref{eq:Lambda2m} respectively 
with $(Q,m)$ replaced by $(Q^{\mu},2)$. 
Noting the estimate \eqref{eq:H153} for 
$(Q^{\ep}, Q_0^{\ep})$
also holds with $\mu$ in place of $\ep$, 
we can take sufficiently large positive constants 
$A=A(T,\|Q_0\|_{H^m},m)$ 
and 
$C_1=C_1(T,\|Q_0\|_{H^m},m)$ 
to satisfy 
\begin{align}
\frac{1}{C_1}
\|W(t)\|_{H^2}^2
&\leqslant 
\mathcal{E}_2^{\mu, \nu}(W(t))
\leqslant 
C_1\,\|W(t)\|_{H^2}^2
\quad 
\text{on $[0,T]$.}
\label{eq:92415}
\end{align}
In what follows, we derive the estimate for 
$\mathcal{E}_2^{\mu, \nu}(W(t))$. 
\par 
For this purpose, 
we first set $Z=\p_x^2W
+M^{-1}\Lambda(Q^{\mu})\overline{W}
-M^{-1}\Lambda_{2,2}(Q^{\mu})W$, 
and set
\[
\mathcal{Z} :=
\begin{pmatrix}
	Z \ \ 
	\overline{Z} 
\end{pmatrix}^{\top},
	\quad 
\mathcal{W}_2 :=
\begin{pmatrix}
	\p_x^2W \ \
	\overline{\p_x^2W}
\end{pmatrix}^{\top}. 
\]
If we use the setting of Section~\ref{subsec:reduction}, 
then $\mathcal{Z}$ can be expressed as  
$$
\mathcal{Z}
=
\left(
I+\mathcal{L}(Q^{\mu})\p_x^{-2}-\mathcal{L}_{2,2}(Q^{\mu})\p_x^{-2}
\right)
\mathcal{W}_2. 
$$
Hence the same computation as that 
we get \eqref{eq:W_system} from \eqref{eq:W_def} works, except that 
the term corresponding to $\mathcal{L}_{1,2}(Q)\p_x^{-3}$ is dropped here. 
Taking the difference in account, and  using $M=a\,I_n$, and structural conditions
($\widetilde{C1}$)-($\widetilde{C2}$) with $m=2$ for $Q^{\mu}$, 
we can obtain 
\begin{align}
\p_t\mathcal{Z}
&=
(-\mu^5\,I+\mathcal{M})\p_x^4\,\mathcal{Z}
+\p_x\left\{\mathcal{R}_S(Q^{\mu})\p_x\mathcal{Z}\right\}
+\mathcal{H}_2(Q^{\mu},\p_xQ^{\mu})\p_x\mathcal{Z}
\nonumber
\\
&\quad
+ 
\mathcal{K}_{2,S}^{\mathrm{off}}(Q^{\mu}, \p_xQ^{\mu})\p_x\mathcal{Z}
+
(\nu^5-\mu^5)(\p_x^6Q^{\nu} \ \overline{\p_x^6Q^{\nu}})^{\top}
+\mu^5O(|Q^{\mu}||\p_xQ^{\mu}|)\p_x\mathcal{Z}
\nonumber
\\
&\quad 
+(r_2 \ \overline{r_2})^{\top}, 
\nonumber
\end{align}
where 
$$
\|r_2(t)\|_{L^2}
\leqslant 
C(\|Q^{\mu}(t)\|_{H^4}, \|Q^{\nu}(t)\|_{H^4}) \|W(t)\|_{H^2} 
\quad \text{for any $t\in [0,T]$.}
$$
Note that the seemingly problematic first-order term 
$\mathcal{K}_{2,S}^{\mathrm{off}}(Q^{\mu}, \p_xQ^{\mu})\p_x\mathcal{Z}$
still remains. 
Furthermore, using the first $n$-components of the system, 
we have
\begin{align}
	\p_t Z
	&= (-\mu^5I_n+ i\,M )\p_x^4 Z
	+ \p_x
	\left\{
	J(Q^{\mu})\p_x Z
	+K(Q^{\mu})\overline{\p_x Z}
	\right\}
	\nonumber
	\\&\quad
		+ H_2(Q^{\mu},\p_xQ^{\mu})\p_x Z 
	+S_2(Q^{\mu},\p_xQ^{\mu})\overline{\p_x Z}
	\nonumber
	\\
	&\quad 
	+
	\begin{pmatrix}
		I_n & 0_n
	\end{pmatrix}
	\mathcal{K}_{2,S}^{\mathrm{off}}(Q^{\mu},\p_xQ^{\mu})
	\begin{pmatrix}
		0_n & I_n
	\end{pmatrix}^{\top}\overline{\p_x Z}
	\nonumber
	\\
	&\quad 
	+(\nu^5-\mu^5)\p_x^6Q^{\nu}
	+\mu^5
	O(|Q^{\mu}||\p_xQ^{\mu}||\p_xZ|)
	+r_2, 
	\nonumber
\end{align}
where $J_2(Q^{\mu})$, $K_2(Q^{\mu})$, 
$H_2(Q^{\mu},\p_xQ^{\mu})$, 
$S_2(Q^{\mu},\p_xQ^{\mu})$, 
and $\mathcal{K}_{2,S}^{\mathrm{off}}(Q^{\mu},\p_xQ^{\mu})$ 
have been determined below \eqref{eq:sys_Wep} with 
$(Q^{\ep}, m)$ replaced by $(Q^{\mu}, 2)$. 
By using the properties of these matrix-valued functions 
and \eqref{eq:92415}, 
we can deduce 
\begin{align}
&\dfrac{d}{dt}
\left[
\frac{1}{2}
\|
\p_x^2W(t)
+M^{-1}\Lambda(Q^{\mu})\overline{W}(t)
-M^{-1}\Lambda_{2,2}(Q^{\mu})W(t)
\|_{L^2}^2
\right]
=\operatorname{Re}
\langle
\p_tZ,Z
\rangle
\nonumber
\\
&\leqslant 
-\dfrac{\mu^5}{2}\|\p_x^2Z\|_{L^2}^2
+C(\|Q^{\mu}(t)\|_{H^4}, \|Q^{\nu}(t)\|_{H^4})
\mathcal{E}_2^{\mu,\nu}(W(t))
+(\nu^5-\mu^5)\|Q^{\nu}(t)\|_{H^6}\|Z(t)\|_{L^2}
\nonumber
\\
&\quad 
+\operatorname{Re}
\langle
\begin{pmatrix}
	I_n & 0_n
\end{pmatrix}
\mathcal{K}_{2,S}^{\mathrm{off}}(Q^{\mu},\p_xQ^{\mu})
\begin{pmatrix}
	0_n & I_n
\end{pmatrix}^{\top}\overline{\p_x Z}
,Z
\rangle. 
\label{eq:6122}
\end{align}
\par 
On the other hand, after lengthy computations, 
we can obtain 
\begin{align}
&\dfrac{d}{dt}\left[
-\operatorname{Re}
\langle
\p_xW(t), M^{-1}\Lambda_{1,2}(Q^{\mu},\p_xQ^{\mu})
\overline{W}(t)
\rangle
\right]
\nonumber
\\
&\leqslant 
\dfrac{\mu^5}{4}
\|\p_x^4W\|_{L^2}^2
+C(\|Q^{\mu}(t)\|_{H^5}, \|Q^{\nu}(t)\|_{H^4})
\|W(t)\|_{H^2}^2
\nonumber
\\
&\quad 
+C(T,\|Q_0\|_{H^m},m)
(\nu^5-\mu^5)\|Q^{\nu}(t)\|_{H^5}\|W(t)\|_{H^1}
\nonumber
\\
&\quad 
-\dfrac{1}{2}\operatorname{Re}
\left\langle
\left((B_2-B_2^{\top})(Q^{\mu},\p_xQ^{\mu})
-\p_x
(P_2-P_2^{\top})(Q^{\mu})
\right)
\overline{\p_x Z}
,Z
\right\rangle. 
\label{eq:6123}
\end{align}
Although we omit the details 
because the computations are similar to those 
used to derive (4.19) in \cite{Onodera2025},
we remark that the assumption $m\geqslant 5$ 
is used in deriving the above estimate.
Moreover, the last terms of \eqref{eq:6122} and \eqref{eq:6123}
cancel each other out, since 
\begin{align}
\mathcal{K}_{2,S}^{\mathrm{off}}
&=
\dfrac{1}{2}
\begin{pmatrix}
0_n & B_2-B_2^{\top} \\
\overline{B_2-B_2^{\top}} & 0_n
\end{pmatrix}
-\dfrac{1}{2}
\begin{pmatrix}
0_n & \p_x(P_2-P_2^{\top}) \\
\overline{\p_x(P_2-P_2^{\top})} & 0_n
\end{pmatrix}
\nonumber
\end{align}
by the definition of $\mathcal{K}(Q,\p_xQ)$ 
which is given by \eqref{eq:K_m}. 
This relation will be explained in more detail in 
Section~\ref{sec:LWP4}
(see \eqref{eq:KmSoff}). 
\par 
Combining \eqref{eq:6122}-\eqref{eq:6123} and the estimate 
of  time-derivative of $A\|W(t)\|_{H^1}^2$, and using 
\eqref{eq:92415}, 
we derive 
\begin{align}
	\dfrac{d}{dt}
	\mathcal{E}_2^{\mu,\nu}(W(t))
	&\leqslant 
	C(T,\|Q_0\|_{H^m},m)(\nu^5-\mu^5)
	\|Q^{\nu}(t)\|_{H^6}\|W(t)\|_{H^1}
	\nonumber
	\\
	&\quad 
	+C(T, \|Q_0\|_{H^m}, m,\|Q^{\mu}(t)\|_{H^5}, \|Q^{\nu}(t)\|_{H^4})
	\mathcal{E}_2^{\mu,\nu}(W(t)). 
	\nonumber
\end{align}
Here, by the first inequality of \eqref{eq:H153} for $m=6$ 
and \eqref{eq:6092} for $(\ep,m,\ell)=(\nu,5,1)$,  
$$
\|Q^{\nu}\|_{C([0,T];H^6)}
\leqslant 
C(T,\|Q_0\|_{H^5})\nu^{-1}\|Q_0\|_{H^5}.
$$ 
Using this, \eqref{eq:92415} and \eqref{eq:H153}, 
we have   
\begin{align}
\dfrac{d}{dt}
\mathcal{E}_2^{\mu,\nu}(W(t))
&\leqslant 
C(T, \|Q_0\|_{H^m},m)
\left\{
\mathcal{E}_2^{\mu,\nu}(W(t))
+
\nu^4
\left(\mathcal{E}^{\mu,\nu}_2(W(t))\right)^{1/2}
\right\}. 
\label{eq:u385}
\end{align} 
The Gronwall inequality and \eqref{eq:92415} implies   
\begin{align}
\|W(t)\|_{H^2}
&\leqslant
C(T,\|Q_0\|_{H^m},m)
(\|W(0)\|_{H^2}+\nu^4). 
\nonumber
\end{align}
By the triangle inequality 
$\|W(0)\|_{H^2}=
\|Q^{\mu}_0-Q^{\nu}_0\|_{H^2}
\leqslant 
\|Q^{\mu}_0-Q_0\|_{H^2}
+\|Q_0-Q^{\nu}_0\|_{H^2}$, 
\eqref{eq:6093} with $\ell=m-2$ and $\ep=\mu,\nu$, 
and by $0<\mu\leqslant \nu<1$, 
we derive 
\begin{align}
\|W\|_{C([0,T];H^2)}
&
\leqslant 
2C(T,\|Q_0\|_{H^m},m)(\nu^{m-2}+\nu^4),  
\label{eq:90312}
\end{align}
which is the desired \eqref{eq:1cauchy}. 
\end{proof}
\begin{remark} 
	\label{remark:ME}
	By analogy with   
		\eqref{eq:Wvar} with $(W,2)$ in place of 
		$(\p_x^mQ^{\ep}, m)$, 
	it is natural that the 
	following function is expected to play an essential 
	role to eliminate the loss of derivatives arising from 
	the system satisfied by $\p_x^2W$:  
	$$\p_x^2W
	+M^{-1}\Lambda(Q^{\mu})\overline{W}
	+M^{-1}\Lambda_{1,2}(Q^{\mu},\p_xQ^{\mu})
	\overline{\p_x^{-1}W}
	-M^{-1}\Lambda_{2,2}(Q^{\mu})W. 
	$$  
	However, the term $\p_x^{-1}W$ in it 
	is not well-defined.    
	On the other hand, a formal computation using integration by parts tells 
	the square of the $L^2$-norm of the above term can be 
	rewritten as
	\begin{align}
		&\left\|
		\p_x^2W
		+M^{-1}\Lambda(Q^{\mu})\overline{W}
		-M^{-1}\Lambda_{2,2}(Q^{\mu})W
		\right\|_{L^2}^2
		\nonumber
		\\
		&\quad 
		-2\operatorname{Re}
		\langle
		\p_xW, M^{-1}\Lambda_{1,2}(Q^{\mu},\p_xQ^{\mu})
		\overline{W}
		\rangle
		+\text{(terms including $\p_x^{-1}W$).}
		\nonumber 
	\end{align}
	We therefore introduce
	$\mathcal E_2^{\mu,\nu}(W)$
	by retaining only the first two terms
	on the right-hand side, which are well-defined. 
	The addition of $A\|W\|_{H^1}^2$, 
	which is motivated by the idea on the choice of  
	modified energies in \cite{Kwon,Segata2012}, 
	ensures the 
	positivity of  $\mathcal{E}_2^{\mu,\nu}(W)$ and the equivalence property \eqref{eq:92415}. 
\end{remark}
\begin{proof}[Proof of \eqref{eq:mcauchy} in Proposition~\ref{proposition:cauchy}]
For $\mu, \nu$ with 
$0<\mu \leqslant \nu<1$, 
we set $W:=Q^{\mu}-Q^{\nu}: [0,T]\times \TT\to \mathbb{C}^n$ again and 
define 
\begin{align}
W_m^{\mu,\nu} 
&:= \p_x^m W
+M^{-1}\Lambda(Q^{\mu})
\overline{\p_x^{m-2}W} 
+M^{-1}\Lambda_{1,m}(Q^{\mu},\p_xQ^{\mu})
\overline{\p_x^{m-3}W} 
\nonumber
\\
&\qquad 
-M^{-1}\Lambda_{2,m}(Q^{\mu})
\p_x^{m-2}W, 
\label{eq:Wmmn}
\end{align}
where $\Lambda$, $\Lambda_{1,m}$ and  $\Lambda_{2,m}$ are matrix-valued 
functions determined by \eqref{eq:Lambda}, \eqref{eq:Lambda1m}, and \eqref{eq:Lambda2m} 
respectively with $Q$ replaced by $Q^{\mu}$.
We define  
$\mathcal{E}^{\mu,\nu}_m(W)
=\mathcal{E}^{\mu,\nu}_m(W(t))$ 
by 
\begin{align}
\mathcal{E}^{\mu,\nu}_m(W(t))
&=
\sqrt{
\|W_m^{\mu,\nu}(t)\|_{L^2}^2+\|W(t)\|_{H^{m-1}}^2
}. 
\label{eq:Ekl}
\end{align}
There exists  
$C_2(T,\|Q_0\|_{H^m},m)>0$ 
which is independent of $\mu$ and $\nu$ such that 
 \begin{align}
 &\frac{1}{C_2}\|W(t)\|_{H^m}^2
 \leqslant 
 \mathcal{E}^{\mu,\nu}_m(W(t))^2
 \leqslant 
 C_2\|W(t)\|_{H^m}^2
 \quad 
 \text{on $[0,T]$.}
 \label{eq:H4Em}
 \end{align}
 Using $M=a\,I_n$, 
 the structural conditions 
 ($\widetilde{C1}$)- ($\widetilde{C2}$) with $m\geqslant 5$ 
 for $Q^{\mu}$, 
 and using $m\geqslant 5$ as above, 
 it is now not difficult to see the following estimate holds:
\begin{align}
\frac{1}{2}\dfrac{d}{dt}
\|W_m^{\mu,\nu}(t)\|_{L^2}^2
&\leqslant 
C(T,\|Q_0\|_{H^m},m)\mathcal{E}^{\mu,\nu}_m(W(t))^2
\nonumber
\\
&\quad 
+C(T,\|Q_0\|_{H^m},m)
\|W(t)\|_{H^1}
\|Q^{\nu}(t)\|_{H^{m+2}}
\|W_m^{\mu,\nu}(t)\|_{L^2}
\nonumber
\\
&\quad
+C(T,\|Q_0\|_{H^m},m)
(\nu^5-\mu^5)
\|Q^{\nu}(t)\|_{H^{m+4}}
\|W_m^{\mu,\nu}(t)\|_{L^2}.
\label{eq:9016}
\end{align} 
Further, by combining the first inequality of 
\eqref{eq:H153} with $m$ replaced by $m+j$ for 
$j=1,2,\ldots$ and \eqref{eq:6092},  
\begin{align}
\|Q^{\nu}\|_{C([0,T];H^{m+j})}
&\leqslant 
C(T,\|Q_0\|_{H^m},m)
\nu^{-j}
\quad 
(j=1,2,\ldots).
\label{eq:90162}
\end{align}
Combining \eqref{eq:90162} for $j=2,4$ and  
\eqref{eq:90312}, 
and then using \eqref{eq:H4Em} and $\nu\in (0,1)$, 
we derive 
\begin{align}
\frac{d}{dt}\|W_{m}^{\mu,\nu}(t)\|_{L^2}^2
\leqslant 
C(T,\|Q_0\|_{H^m},m)
\left\{
\mathcal{E}^{\mu,\nu}_m(W(t))^2
+
(\nu^{m-4}+\nu)
\mathcal{E}^{\mu,\nu}_m(W(t))
\right\}. 
\nonumber
\end{align} 
Combining this and the estimate for 
the time-derivative of $\|W(t)\|_{H^{m-1}}^2$, 
we obtain 
\begin{align}
&\frac{d}{dt}
\mathcal{E}^{\mu,\nu}_m(W)^2
\leqslant 
C(T,\|Q_0\|_{H^m},m)
\left\{
\mathcal{E}^{\mu,\nu}_m(W)^2
+
(\nu^{m-4}+\nu)
\mathcal{E}^{\mu,\nu}_m(W)
\right\}
\nonumber 
\end{align} 
holds on $[0,T]$. 
By the Gronwall inequality 
and \eqref{eq:H4Em}, 
this shows   
$$
\|W\|_{C([0,T];H^m)}
\leqslant 
C(T,\|Q_0\|_{H^m},m)
\left(
\nu^{m-4}
 +\nu
+\|W(0)\|_{H^m}
\right),  
$$
which is the desired \eqref{eq:mcauchy}.
\end{proof}
\par 
Once Proposition~\ref{proposition:cauchy} is derived, 
it is straightforward to complete 
the proof of local well-posedness 
in the same way as that in \cite{Onodera2025}: 
By \eqref{eq:mcauchy}, 
a time-local solution is constructed 
as the limit of $\left\{Q^{\ep}\right\}_{\ep\in (0,1)}$ 
in $C([0,T];H^m(\TT;\CC^n))$.  
The uniqueness is proved by the estimate for 
the modified energy 
\begin{align}
\mathcal{E}_2(W(t))
&=
\frac{1}{2}
\|
\p_x^2W(t)
+M^{-1}\Lambda(Q^{1})\overline{W}(t)
-M^{-1}\Lambda_{2,2}(Q^{1})W(t)
\|_{L^2}^2
\nonumber
\\
&\quad 
-\operatorname{Re}
\langle
\p_xW(t), M^{-1}\Lambda_{1,2}(Q^{1},\p_xQ^{1})
\overline{W}(t)
\rangle
+A\, \|W(t)\|_{H^1}^2, 
\label{eq:5314}
\end{align}
where $W=Q^1-Q^2$ for solutions $Q^1$ and $Q^2$ 
to \eqref{eq:IVP} with same initial data, 
$\Lambda_1(Q^{1})$, 
$\Lambda_{1,2}(Q^{1},\p_xQ^{1})$ and 
$\Lambda_{2,2}(Q^{1})$ are determined by 
\eqref{eq:Lambda}, \eqref{eq:Lambda1m},   
and \eqref{eq:Lambda2m} respectively 
with $(Q,m)$ replaced by $(Q^{1},2)$, 
and  
$A=A(\|Q^1\|_{C([0,T];H^5)})$ is taken to be sufficiently 
large to ensure the equivalence of  $\mathcal{E}_2(W(t))$ 
and $\|W(t)\|_{H^2}$ on the time-interval $[0,T]$. 
Indeed, by imposing $M=a\,I_n$, 
the structural conditions 
($\widetilde{C1}$)-($\widetilde{C2}$) 
with $m=2$ for $Q^1$,  
the same argument to derive \eqref{eq:90312} 
from \eqref{eq:5313} 
but with $\mu=\nu=0$ works under $m\geqslant 5$, 
which shows $W=0$. 
The data-to-solution map is continuous 
from $H^m$ to $C([0,T'];H^m(\TT;\CC^n))$ for any $T' \in (0,T)$, 
which is proved by a standard 
argument based on \eqref{eq:mcauchy}. 
We omit the detail.
\par 
Summarizing the above discussion, 
we obtain the following proposition. 
\begin{proposition}
\label{pro:framework}
The initial value problem \eqref{eq:IVP} is time-locally well-posed 
in $H^m(\TT;\CC^n)$ for any integer $m\geqslant 5$ in the sense of 
Theorem~\ref{thm:main}, 
provided that $M=a\,I_n$ for some 
$a\in \mathbb{R} \setminus \{0\}$, 
the conditions ($\widetilde{C1}$)-($\widetilde{C2}$) 
with $m\geqslant 5$
for  $Q^{\ep}$ and $Q^{\mu}$,  
and the same conditions  with $m=2$ 
for $Q^{\mu}$ and $Q^1$ 
hold in the above discussion in this subsection.  
\end{proposition}  
\section{Structural conditions (A2) and (C1)--(C2) to 
prove Theorem~\ref{thm:main} under $m\geqslant 5$}
\label{sec:A2C1C2}
In this section, we verify that the structural conditions 
(A2) and (C1)--(C2) ensure the key properties 
introduced in Section~\ref{sec:proof}, 
thereby completing the proof of 
Theorem~\ref{thm:main}.
\begin{proof}[Proof of Theorem~\ref{thm:main}]
Proposition~\ref{pro:framework} assumes the conditions 
($\widetilde{C1}$)-($\widetilde{C2}$) 
with $m\geqslant 5$
for $Q^{\ep}$ and $Q^{\mu}$,  
and with $m=2$ for $Q^{\mu}$ and $Q^1$.  
Thus, it suffices to show that these requirements are ensured 
by the structural conditions
 (A2) and (C1)--(C2). 
\par 
In what follows, we use the notation 
$(v)_j$ to denote the $j$-th component of 
$\CC^n$-valued functions $v=v(t,x)$, 
and express $n\times n$ matrix-valued functions $A=A(t,x)$ 
by $A=((A)_{jr})_{1\leqslant j,r\leqslant n}$ 
where $(A)_{jr}$ denotes the $(j,r)$-component of $A$.   
\par 
We first investigate the relationship between 
($\widetilde{C1}$) and (C1). 
For this purpose, we obtain 
the exact expression 
of $\mathcal{R}_H^{\mathrm{diag}}(Q)$ for 
$\CC^n$-valued functions $Q$ on 
$I_T\times \TT$.   
Recalling the definition of $\mathcal{R}(Q)$ 
which is given below 
\eqref{eq:matrix_system}
and the skew-Hermitian property $M_S^{*}=-M_S$, 
we obtain 
\begin{align}
\mathcal{R}_H(Q)
&=
\dfrac{1}{2}(\mathcal{R}(Q) +\mathcal{R}^{*}(Q) )
\nonumber
\\
&=
\dfrac{1}{2}
\left\{
\begin{pmatrix}
M_S+P_1(Q) & P_2(Q) \\
\overline{P_2(Q)} & \overline{M_S+P_1(Q)}
\end{pmatrix}
+
\begin{pmatrix}
M_S^*+P_1^*(Q) & \overline{P_2^*(Q)} \\
P_2^*(Q) & \overline{M_S^*+P_1^*(Q)}
\end{pmatrix}
\right\}
\nonumber
\\
&=
\dfrac{1}{2}
\begin{pmatrix}
P_1(Q)+P_1^*(Q) & P_2(Q)+P_2^{\top}(Q) \\
\overline{P_2(Q)+P_2^{\top}(Q)} & \overline{P_1(Q)+P_1^*(Q)}
\end{pmatrix}, 
\label{eq:Rh}
\\
\mathcal{R}_H^{\mathrm{diag}}(Q)
&=
\dfrac{1}{2}
\begin{pmatrix}
P_1(Q)+P_1^*(Q) & 0_n \\
0_n & \overline{P_1(Q)+P_1^*(Q)}
\end{pmatrix}.
\label{eq:imp1}
\end{align}
The $j$-th component of 
$P_1(Q)\p_x^2U$, which is written as $(P_1(Q)\p_x^2U)_j$ 
in our rule of notation, 
equals the 
term which includes $\p_x^2U$
in $\p_x^m
\left\{
F_j(Q,\p_xQ,\p_x^2Q)
\right\}$, 
which is nothing but  
\[
\sum_{p,q,r=1}^n
\omega_{p,q,r}^{1,j}\p_x^2U_p\overline{Q_q}Q_r
=
\sum_{p,q,r=1}^n
\omega_{r,q,p}^{1,j}Q_p\overline{Q_q}\p_x^2U_r. 
\]
Hence we obtain 
\begin{equation}
P_1(Q)
=
\left(\sum_{p,q=1}^n
\omega_{r,q,p}^{1,j}Q_p\overline{Q_q}
\right)_{1\leqslant j,r\leqslant n}. 
\label{eq:P1}
\end{equation} 
By \eqref{eq:imp1}, 
the condition ($\widetilde{C1}$) 
is equivalent to $0_n=P_1(Q)+P_1^*(Q)$, 
that is
\begin{align}
(\widetilde{C1})
\iff 
0_n
&=
\left(\sum_{p,q=1}^n
\omega_{r,q,p}^{1,j}Q_p\overline{Q_q}
\right)_{1\leqslant j,r\leqslant n}
+
\left(\sum_{p,q=1}^n
\overline{\omega_{j,q,p}^{1,r}}\overline{Q_p}Q_q
\right)_{1\leqslant j,r\leqslant n}
\nonumber
\\
&=
\left(\sum_{p,q=1}^n
\left(\omega_{r,q,p}^{1,j}
+
\overline{\omega_{j,p,q}^{1,r}}
\right)
Q_p\overline{Q_q}
\right)_{1\leqslant j,r\leqslant n}.   
\nonumber
\end{align}
This shows that ($\widetilde{C1}$), 
which eliminates the problematic second-order terms, 
is ensured by (C1).
\par 
We next investigate the relationship between 
($\widetilde{C2}$) and (C2). 
To this end, we recall \eqref{eq:K_m}.
Since the second term on the right-hand side
is already expressed as a spatial derivative
of $\mathcal{R}(Q)$,
the condition ($\widetilde{C2}$) is 
equivalent to that 
$\begin{pmatrix}
	I_n & 0_n
\end{pmatrix}
\mathcal{R}_{m,S}^{\mathrm{diag}}(Q,\p_xQ)
\begin{pmatrix}
	I_n & 0_n
\end{pmatrix}^{\top}$
is expressed as a spatial derivative of an 
$n\times n$ matrix-valued function of order $|Q|^2$ 
defined on 
$I_T\times \TT$. 
The definition of $\mathcal{R}_{m}=\mathcal{R}_{m}(Q,\p_xQ)$ 
and 
$\mathcal{R}_{m,S}^{\mathrm{diag}}
=\mathcal{R}_{m,S}^{\mathrm{diag}}(Q,\p_xQ)$ 
yields  
\begin{align}
\mathcal{R}_{m,S}
&
=\dfrac{1}{2}(\mathcal{R}_{m}-\mathcal{R}_{m}^*)
=
\dfrac{1}{2}
\left\{
\begin{pmatrix}
A_m & B_m \\
\overline{B_m} & \overline{A_m}
\end{pmatrix}
-\begin{pmatrix}
A_m^* &  \overline{B_m^*}\\
B_m^* & \overline{A_m^*}
\end{pmatrix}
\right\}
\nonumber
\\
&=
\dfrac{1}{2}
\begin{pmatrix}
A_m-A_m^* & B_m-B_m^{\top} \\
\overline{B_m-B_m^{\top}} & \overline{A_m-A_m^*}
\end{pmatrix},  
\label{eq:Rms}
\\
\mathcal{R}_{m,S}^{\mathrm{diag}}
&=
\dfrac{1}{2}
\begin{pmatrix}
A_m-A_m^* & 0_n \\
0_n & \overline{A_m-A_m^*}
\end{pmatrix},   
\nonumber
\end{align}
from which we have 
\begin{equation}
\begin{pmatrix}
	I_n & 0_n
\end{pmatrix}
\mathcal{R}_{m,S}^{\mathrm{diag}}(Q,\p_xQ)
\begin{pmatrix}
	I_n & 0_n
\end{pmatrix}^{\top}
=\dfrac{1}{2}(A_m(Q,\p_xQ)-A_m^*(Q,\p_xQ)).
\label{eq:6181}
\end{equation}
Since $(A_m(Q,\p_xQ)\p_xU)_j$ equals the sum of 
the terms in $\p_x^m\left\{F_j(Q,\p_xQ,\p_x^2Q)\right\}$ 
which include $\p_xU$,  
\begin{align}
(A_m(Q,\p_xQ)\p_xU)_j
&=
m\,
\sum_{p,q,r=1}^n
\omega^{1,j}_{p,q,r}\, \p_x U_p\, 
\p_x(\overline{Q_q}\, Q_r)
+ 
\sum_{p,q,r=1}^n
\omega^{3,j}_{p,q,r}\, \p_x U_p\, \overline{\p_x Q_q}\, Q_r
\nonumber
\\
&\quad 
+ 
\sum_{p,q,r=1}^n
\omega^{4,j}_{p,q,r}\, \p_x U_p\, \p_xQ_q\, \overline{Q_r}
+ 
\sum_{p,q,r=1}^n
\omega^{4,j}_{p,q,r}\, \p_x Q_p\, \p_xU_q\, \overline{Q_r}
\nonumber
\\
&=
m\,
\sum_{p,q,r=1}^n
\omega^{1,j}_{r,q,p}\,  
\p_x(Q_p\, \overline{Q_q})
\,\p_x U_r
+ 
\sum_{p,q,r=1}^n
\omega^{3,j}_{r,q,p}
\, Q_p\, \overline{\p_x Q_q}
\, \p_x U_r
\nonumber
\\
&\quad 
+ 
\sum_{p,q,r=1}^n
(\omega^{4,j}_{r,p,q}+\omega^{4,j}_{p,r,q})
\, \p_xQ_p\,\overline{Q_q}
\,\p_x U_r.
\nonumber
\end{align} 
This shows 
\begin{align}
(A_m(Q,\p_xQ))_{jr}
&=
\sum_{p,q=1}^n
\left\{
m\,\omega^{1,j}_{r,q,p}\,  
\p_x(Q_p\, \overline{Q_q})
+
\omega^{3,j}_{r,q,p}
\, Q_p\, \overline{\p_x Q_q}
+
(\omega^{4,j}_{r,p,q}+\omega^{4,j}_{p,r,q})
\, \p_xQ_p\,\overline{Q_q}
\right\}. 
\nonumber
\end{align}
Therefore, a simple computation yields 
\begin{align}
&(A_m(Q,\p_xQ)-A_m^*(Q,\p_xQ))_{jr}
\nonumber
\\
&=
m\,\sum_{p,q=1}^n
\omega^{1,j}_{r,q,p}\,  
\p_x(Q_p\, \overline{Q_q})
-
m\,\sum_{p,q=1}^n
\overline{\omega^{1,r}_{j,q,p}}\,  
\p_x(\overline{Q_p}\, Q_q)
\nonumber
\\
&\quad
+
\sum_{p,q=1}^n
\omega^{3,j}_{r,q,p}
\, Q_p\, \overline{\p_x Q_q}
-\sum_{p,q=1}^n
\overline{\omega^{3,r}_{j,q,p}}
\, \overline{Q_p}\, \p_x Q_q
\nonumber
\\
&\quad 
+\sum_{p,q=1}^n
(\omega^{4,j}_{r,p,q}+\omega^{4,j}_{p,r,q})
\, \p_xQ_p\,\overline{Q_q}
-\sum_{p,q=1}^n
(\overline{\omega^{4,r}_{j,p,q}}+
\overline{\omega^{4,r}_{p,j,q}})
\, \overline{\p_xQ_p}\,Q_q
\nonumber 
\\
&=
\p_x\left(
m
\,\sum_{p,q=1}^n
(\omega^{1,j}_{r,q,p}-
\overline{\omega^{1,r}_{j,p,q}}
)
\,  
Q_p\, \overline{Q_q}
\right)
+
\sum_{p,q=1}^n
(\omega^{3,j}_{r,q,p}
-\overline{\omega^{4,r}_{j,q,p}}
-
\overline{\omega^{4,r}_{q,j,p}}
)
\, Q_p\, \overline{\p_x Q_q}
\nonumber
\\
&\quad 
-\sum_{p,q=1}^n
(
\overline{\omega^{3,r}_{j,p,q}}
-\omega^{4,j}_{r,p,q}-\omega^{4,j}_{p,r,q}
)
\p_xQ_p\overline{Q_q}. 
\nonumber
\end{align}
Under the assumption (A2), the sum of the second and the third 
terms of the right-hand side becomes 
\[
\sum_{p,q=1}^n
(\omega^{3,j}_{r,q,p}
-2\overline{\omega^{4,r}_{j,q,p}}
)
\, Q_p\, \overline{\p_x Q_q}
-\sum_{p,q=1}^n
(
\overline{\omega^{3,r}_{j,p,q}}
-2\omega^{4,j}_{r,p,q}
)
\p_xQ_p\overline{Q_q}.
\]
Furthermore, under the additional assumption (C2), 
this is rewritten as 
\[
\sum_{p,q=1}^n
(\omega^{3,j}_{r,q,p}
-2\overline{\omega^{4,r}_{j,q,p}}
)
\, 
(Q_p\, \overline{\p_x Q_q}
+
\p_xQ_p\overline{Q_q}
)
=
\p_x\left(
\sum_{p,q=1}^n
(\omega^{3,j}_{r,q,p}
-2\overline{\omega^{4,r}_{j,q,p}}
)
Q_p\overline{Q_q}
\right).
\]
Thus, 
$\begin{pmatrix}
	I_n & 0_n
\end{pmatrix}
\mathcal{R}_{m,S}^{\mathrm{diag}}(Q,\p_xQ)
\begin{pmatrix}
	I_n & 0_n
\end{pmatrix}^{\top}$
is expressed as the spatial derivative of
\begin{equation}
\left(
\dfrac{1}{2}
\,\sum_{p,q=1}^n
\left\{
m
(\omega^{1,j}_{r,q,p}-
\overline{\omega^{1,r}_{j,p,q}}
)
+(\omega^{3,j}_{r,q,p}
-2\overline{\omega^{4,r}_{j,q,p}}
)
\right\}
\,  
Q_p\, \overline{Q_q}
\right)_{1\leqslant j,r\leqslant n}, 
\label{eq:RmSdiagp}
\end{equation}
which is actually an $n\times n$ matrix-valued function 
of order $|Q|^2$
on $I_T\times \TT$. 
This concludes that 
($\widetilde{C2}$) with $m\geqslant 5$ 
and with $m=2$ 
for any of $Q^{\ep}, Q^{\mu}$, and $Q^1$
are ensured by (A2) and (C2).
\end{proof}
\section{Structural conditions from the viewpoint of linear dispersive systems}
\label{sec:linear}
In this section, we interpret the structural conditions 
($\widetilde{C1}$)-($\widetilde{C2}$) 
from the viewpoint of the $L^2$ theory for linear dispersive systems. 
\par 
The following initial value problem for 
a system of linear fourth-order dispersive partial differential equations 
was investigated by Chihara in \cite{Chihara2015}:  
\begin{alignat}{2}
	\left(I_2\p_t+iP\right)\vec{u}
	& =
	\vec{f}(t,x) 
	& \quad
	& \text{in}
	\quad
	\mathbb{R}\times\mathbb{T},
	\label{eq:lpde}      
	\\
	\vec{u}(0,x)
	& =
	\vec{\phi}(x)
	& \quad
	& \text{in}
	\quad
	\mathbb{T},
	\label{eq:ldata}
\end{alignat}
where 
$$
P= ED_x^4+A(x)D_x^3+B(x)D_x^2+C(x)D_x+D(x), 
$$
$\vec{u}(t,x)=(u_1(t,x),u_2(t,x))^{\top}$ 
is a $\mathbb{C}^2$-valued unknown function of 
$(t,x)\in\mathbb{R}\times\mathbb{T}$, 
$D_x=-i\p_x$, 
$\vec{\phi}(x)=(\phi_1(x),\phi_2(x))^{\top}$ 
and 
$\vec{f}(t,x)=(f_1(t,x),f_2(t,x))^{\top}$ 
are given functions, 
$$
E
=
\begin{pmatrix}
	1 & 0 \\ 0 & -1
\end{pmatrix},
\quad 
A(x)
=
\begin{pmatrix}
	a_{11}(x) & a_{12}(x) \\ a_{21}(x) & a_{22}(x)
\end{pmatrix},
\quad
B(x)
=
\begin{pmatrix}
	b_{11}(x) & b_{12}(x) \\ b_{21}(x) & b_{22}(x)
\end{pmatrix},
$$
$$
C(x)
=
\begin{pmatrix}
	c_{11}(x) & c_{12}(x) \\ c_{21}(x) & c_{22}(x)
\end{pmatrix},
\quad
D(x)
=
\begin{pmatrix}
	d_{11}(x) & d_{12}(x) \\ d_{21}(x) & d_{22}(x)
\end{pmatrix},
$$
$a_{jk}(x), b_{jk}(x), c_{jk}(x) \in C^\infty(\mathbb{T})$, 
and 
$C^\infty(\mathbb{T})$ is the set of all complex-valued smooth functions on $\mathbb{T}$. 
He characterized the necessary and sufficient conditions on $A(x), B(x)$ and on $C(x)$ 
for \eqref{eq:lpde}-\eqref{eq:ldata} 
to be $L^2$-well-posed. 
Restricting to the case $A(x) \equiv 0$ with vanishing third-order terms, 
the necessary and sufficient conditions reduce to the following form:
\begin{equation}
	\operatorname{Im}\int_{\TT}b_{11}(x)\,dx=\operatorname{Im}\int_{\TT}b_{22}(x)\,dx
	=\operatorname{Im}\int_{\TT}c_{11}(x)\,dx=\operatorname{Im}\int_{\TT}c_{22}(x)\,dx=0. 
	\label{eq:NScond}
\end{equation}
\par 
On the other hand, \eqref{eq:lpde} under the assumptions
$A(x)\equiv 0$ and $D(x)\equiv 0$ can be 
rewritten as a system 
satisfied by $U(t,x):=(u_2(t,x),\overline{u_1(t,x)})^{\top}$, 
which is formulated by 
\[
\left(I_2\p_t-iI_2\p_x^4 \right)U
-B_1(x)\p_x^2U
-B_2(x)\overline{\p_x^2U}
+C_1(x)\p_x U
+C_2(x)\overline{\p_x U}
 =
F(t,x), 
\]
where 
\begin{align}
B_1(x)
&=
\begin{pmatrix}
		i\, b_{22}(x) & 0 \\ 0 & \overline{i\, b_{11}(x)}
\end{pmatrix}, 
\quad 
B_2(x)
=
\begin{pmatrix}
	0 & i\, b_{21}(x)  \\ \overline{i\, b_{12}(x)} & 0
\end{pmatrix},
\nonumber
\\
C_1(x)
&=
\begin{pmatrix}
	c_{22}(x) & 0 \\ 0 & \overline{c_{11}(x)}
\end{pmatrix}, 
\quad 
C_2(x)
=
\begin{pmatrix}
	0 &  c_{21}(x)  \\ \overline{c_{12}(x)} & 0
\end{pmatrix},
\nonumber
\end{align}
and 
$F(t,x)=(f_2(t,x), \overline{f_1(t,x)})^{\top}$. 
By coupling the system with its complex conjugate, 
we have the following $4$-component system 
\begin{equation}
	\label{eq:matrix_l_system}
	\left(I_4\p_t
	-\begin{pmatrix}
		iI_2 & 0_n \\
		0_n & \overline{iI_2}
	\end{pmatrix}\p_x^4
	- \mathcal{B}(x) \p_x^2
	+ \mathcal{C}(x) \p_x
	\right) \begin{pmatrix}
		U \\
		\overline{U}
	\end{pmatrix}
	=
	\begin{pmatrix}
			F \\
			\overline{F}
		\end{pmatrix} ,
\end{equation}
where 
\[
\mathcal{B}(x) 
=
\begin{pmatrix}
	B_1(x) & B_2(x) \\
	\overline{B_2(x)} & \overline{B_1(x)}
\end{pmatrix},
\quad 
\mathcal{C}(x)
=
\begin{pmatrix}
	C_1(x) & C_2(x) \\
	\overline{C_2(x)} & \overline{C_1(x)}
\end{pmatrix}.
\]
Moreover, 
\begin{align}
\mathcal{B}_H^{\mathrm{diag}}(x)
&=\dfrac{1}{2}
\begin{pmatrix}
	B_1(x)+B_1(x)^* & 0_2 \\
	0_2 & \overline{B_1(x)+B_1(x)^*}
\end{pmatrix}, 
\nonumber
\\
B_1(x)+B_1(x)^*
&=
\begin{pmatrix}
	-2\operatorname{Im}b_{22}(x) & 0 \\
	0 & -2\operatorname{Im}b_{11}(x)
\end{pmatrix}, 
\nonumber
\\
	\mathcal{C}_S^{\mathrm{diag}}(x)
	&=\dfrac{1}{2}
	\begin{pmatrix}
		C_1(x)-C_1(x)^* & 0_2 \\
		0_2 & \overline{C_1(x)-C_1(x)^*}
	\end{pmatrix}, 
	\nonumber
	\\
	\begin{pmatrix}
		I_2 & 0_2
	\end{pmatrix}
	\mathcal{C}_{S}^{\mathrm{diag}}(x)
	\begin{pmatrix}
		I_2 & 0_2
	\end{pmatrix}^{\top}
	&=
	\dfrac{C_1(x)-C_1(x)^*}{2}
	=i\,
	\begin{pmatrix}
		\operatorname{Im}c_{22}(x) & 0 \\
		0 & -\operatorname{Im}c_{11}(x)
	\end{pmatrix}. 
	\nonumber
\end{align}
Comparing \eqref{eq:matrix_system} and \eqref{eq:matrix_l_system}, 
we see that ($\widetilde{C1}$) for $n=2$ corresponds to 
\begin{equation} 
\operatorname{Im}b_{11}(x)=\operatorname{Im}b_{22}(x)=0
\label{eq:NScond1}
\end{equation} 
for \eqref{eq:matrix_l_system}, 
and ($\widetilde{C2}$) for $n=2$ corresponds to 
\begin{equation} 
\operatorname{Im}\int_{\TT} c_{11}(x)\,dx
=\operatorname{Im}\int_{\TT} c_{22}(x)\,dx=0
\label{eq:NScond2}
\end{equation} 
for \eqref{eq:matrix_l_system}. 
Moreover, both 
\eqref{eq:NScond1} and \eqref{eq:NScond2} 
satisfy the necessary and sufficient conditions 
\eqref{eq:NScond}. 
This comparison indicates that the elimination of 
derivative loss in the nonlinear system 
corresponds to the cancellation of certain 
lower-order contributions 
in the associated linear dispersive system.
\par 
Additionally, in the setting of the nonlinear problem
\eqref{eq:IVP},
the diagonal blocks of
$\mathcal{R}_H^{\mathrm{diag}}(Q)$,
whenever nonvanishing,
are necessarily of order $|Q|^2$.
Such terms cannot be expressed as spatial derivatives
of periodic functions.
This suggests that it is unlikely that one can construct
a nonlinearity
$F(Q,\p_xQ,\p_x^2Q)$ 
whose associated linear problem 
satisfies \eqref{eq:NScond} but not \eqref{eq:NScond1}
by merely modifying the coefficients,
without introducing additional terms. 
In this sense, although not rigorously proved,
the conditions (C1)--(C2) under (A2)
are expected to be nearly sharp
within the class of nonlinearities
of the form $F(Q,\p_xQ,\p_x^2Q)$
when $M=aI_n$. 
We also note that the corresponding $L^2$-theory
for linear dispersive systems is currently not available
in the case $n\geqslant 3$.
\section{Structural conditions to prove Corollary~\ref{cor:cor} under $m\geqslant 4$} 
\label{sec:LWP4}
In this section, we present the proof of Corollary~\ref{cor:cor}.  
\begin{proof}[Proof of  Corollary~\ref{cor:cor}]
We show that the additional conditions (A1) 
and (D1)--(D2) eliminate the remaining off-diagonal 
contributions responsible for the first-order
derivative loss,
by implying the vanishing of $\mathcal{K}_{m,S}^{\mathrm{off}}=\mathcal{K}_{m,S}^{\mathrm{off}}(Q,\p_xQ)$ 
with $m\geqslant 4$ for $Q=Q^{\ep}$ and $Q=Q^{\mu}$ 
and the same property 
with $m=2$ for $Q=Q^1$ and $Q=Q^{\mu}$.
Under these conditions,   
both $\Lambda_{1,m}$ 
and $\Lambda_{1,2}$
vanish, 
eliminating the need to estimate the terms 
$\p_t(\Lambda_{1,m})\p_x^{-3}$ and   
$\p_t(\Lambda_{1,2})\p_x^{-3}$ 
in the proof of Theorem~\ref{thm:main}. 
Since the assumption $m\geqslant 5$ in Theorem~\ref{thm:main} 
is imposed to handle such terms, 
the absence of these terms naturally yields  
local well-posedness in $H^m$ for $m\geqslant 4$. 
Additionally, the proof  
reduces to the framework demonstrated in
\cite{Onodera2025}.  
\par 
To begin with, 
recall that
$\mathcal{K}_{m,S}
=\mathcal{R}_{m,S}-\p_x(\mathcal{R}_S)$ 
and $\mathcal{K}_{m,S}^{\mathrm{off}}
=
\mathcal{R}_{m,S}^{\mathrm{off}}
-\p_x(\mathcal{R}_S^{\mathrm{off}})$ 
follow from the definition \eqref{eq:K_m}. 
In the same way as we see \eqref{eq:Rh}, we have 
\begin{align}
\mathcal{R}_S
&=
\dfrac{1}{2}(\mathcal{R}-\mathcal{R}^{*} )
=
\dfrac{1}{2}
\begin{pmatrix}
2M_S+P_1-P_1^* & P_2-P_2^{\top} \\
\overline{P_2-P_2^{\top}} & \overline{2M_S+P_1-P_1^*}
\end{pmatrix}. 
\label{eq:Rs}
\end{align}
From \eqref{eq:Rms} and \eqref{eq:Rs}, 
it follows that 
\begin{align}
\mathcal{K}_{m,S}^{\mathrm{off}}
&=
\dfrac{1}{2}
\begin{pmatrix}
0_n & B_m-B_m^{\top} \\
\overline{B_m-B_m^{\top}} & 0_n
\end{pmatrix}
-\dfrac{1}{2}
\begin{pmatrix}
0_n & \p_x(P_2-P_2^{\top}) \\
\overline{\p_x(P_2-P_2^{\top})} & 0_n
\end{pmatrix}. 
\label{eq:KmSoff}
\end{align}
Hence it suffices to obtain the condition 
\begin{equation}
B_m(Q,\p_xQ)-B_m(Q,\p_xQ)^{\top}
-\p_x(P_2(Q)-P_2^{\top}(Q)) =0 
\label{eq:BmdelP2}
\end{equation}
with $m\geqslant 4$ and also with $m=2$ for  
$\CC^n$-valued function $Q$. 
\par 
Since $(P_2(Q)\overline{\p_x^2U})_j$ is the term in 
$\p_x^m\left\{F_j(Q,\p_xQ,\p_x^2Q)\right\}$ 
which includes $\overline{\p_x^2U}$, 
\[
(P_2(Q)\overline{\p_x^2U})_j
=
\sum_{p,q,r=1}^n
\omega^{2,j}_{p,q,r}\, \overline{\p_x^2U_p}\, Q_q\, Q_r
=
\sum_{p,q,r=1}^n
\omega^{2,j}_{r,q,p}\, 
Q_p\, Q_q\,
\overline{\p_x^2U_r}.
\]
Hence, 
\begin{equation}
P_2(Q)
=
\left(
\sum_{p,q=1}^n
\omega^{2,j}_{r,q,p}\, 
Q_p\, Q_q
\right)_{1\leqslant j,r\leqslant n}, 
\label{eq:P2}
\end{equation}
\begin{equation}
\p_x(P_2(Q)-P_2^{\top}(Q))
=
\left(
\sum_{p,q=1}^n
(\omega^{2,j}_{r,q,p}
-
\omega^{2,r}_{j,q,p}
) 
\p_x(Q_p\, Q_q)
\right)_{1\leqslant j,r\leqslant n}. 
\nonumber
\end{equation}
Since 
$(B_m(Q,\p_xQ)\overline{\p_xU})_j$ equals the sum of 
the terms in $\p_x^m\left\{F_j(Q,\p_xQ,\p_x^2Q)\right\}$ 
which include $\overline{\p_xU}$,  
\begin{align}
(B_m(Q,\p_xQ)\overline{\p_xU})_j
&=
m\sum_{p,q,r=1}^n
\omega^{2,j}_{p,q,r}\, \overline{\p_xU_p}\, 
\p_x(Q_q\, Q_r) 
+ 
\sum_{p,q,r=1}^n
\omega^{3,j}_{p,q,r}\, \p_x Q_p\, 
\overline{\p_x U_q}\, Q_r
\nonumber
\\
&=
m\sum_{p,q,r=1}^n
\omega^{2,j}_{r,q,p}\, \p_x(Q_p\, Q_q)
\,\overline{\p_xU_r} 
+ 
\sum_{p,q,r=1}^n
\omega^{3,j}_{p,r,q}\, \p_x Q_p
\, Q_q\, \overline{\p_x U_r}. 
\nonumber
\end{align}
Hence 
\[
B_m(Q,\p_xQ)
=
\left(
m\,\sum_{p,q=1}^n
\omega^{2,j}_{r,q,p}\, \p_x(Q_p\, Q_q) 
+ 
\sum_{p,q=1}^n
\omega^{3,j}_{p,r,q}
\, \p_x Q_p\, Q_q
\right)_{1\leqslant j,r\leqslant n}, 
\] 
\begin{align}
&B_m(Q,\p_xQ)-B_m(Q,\p_xQ)^{\top}
\nonumber
\\
&=
\left(
m\,\sum_{p,q=1}^n
(\omega^{2,j}_{r,q,p}
-
\omega^{2,r}_{j,q,p}
)\, \p_x(Q_p\, Q_q) 
+ 
\sum_{p,q=1}^n
(\omega^{3,j}_{p,r,q}
-
\omega^{3,r}_{p,j,q}
)
\, \p_x Q_p\, Q_q
\right)_{1\leqslant j,r\leqslant n}. 
\nonumber
\end{align}
It follows from 
\begin{align}
&B_m(Q,\p_xQ)-B_m(Q,\p_xQ)^{\top}
-\p_x(P_2(Q)-P_2^{\top}(Q))
\nonumber
\\
&=
\left(
(m-1)\,\sum_{p,q=1}^n
(\omega^{2,j}_{r,q,p}
-
\omega^{2,r}_{j,q,p}
)\, \p_x(Q_p\, Q_q) 
+ 
\sum_{p,q=1}^n
(\omega^{3,j}_{p,r,q}
-
\omega^{3,r}_{p,j,q}
)
\, \p_x Q_p\, Q_q
\right)_{1\leqslant j,r\leqslant n}.
\nonumber
\end{align}
Considering the symmetry of the first summation term 
of the right-hand side including $\p_x(Q_pQ_q)$ with respect 
to the change of indexes $p$ and $q$, we see 
both \eqref{eq:BmdelP2} with $m\geqslant 4$ and 
also with $m=2$ hold if 
$$
(\omega^{2,j}_{r,q,p}
-
\omega^{2,r}_{j,q,p})
+(\omega^{2,j}_{r,p,q}
-
\omega^{2,r}_{j,p,q})
=0
\quad 
\text{
and 
}
\quad 
\omega^{3,j}_{p,r,q}
-
\omega^{3,r}_{p,j,q}=0
$$ 
for all $p,q,r,j\in \left\{1,\ldots,n\right\}$.
Under (A1), the above conditions are equivalent to 
(D1) and (D2), respectively.
\end{proof}
\section{The case $M\neq aI_n$}
\label{sec:M}
In this section, 
we discuss the case 
when $M=\diag(a_1,\ldots,a_n)$ for 
$(a_1,\ldots,a_n) \in (\mathbb{R} \setminus \{0\})^n$ 
without $a_1=\cdots=a_n$, 
and present the proof of Corollary~\ref{cor:cor2}. 
\begin{proof}[Proof of Corollary~\ref{cor:cor2}] 
The main difficulty in this case lies in the loss of 
commutativity between the principal operator 
and the gauge-type transformations.
Recalling the proof of Theorem~\ref{thm:main}, 
we observe that the assumption $M=aI_n$ is used to 
ensure that every
$\Lambda$, $\Lambda_{1,m}$, $\Lambda_{2,m}$ 
and those with $m=2$  
commutes with $M$. 
In the present case $M=\diag(a_1,\ldots,a_n)\neq aI_n$, 
they do not commute with $M$ in general. 
Any nonvanishing commutator then
causes a loss of one or two derivatives
in the energy estimates,
which prevents the proof strategy of 
Theorem~\ref{thm:main} from working.
However, as is discussed in \cite{Onodera2025}, 
our proof strategy remains valid as far as  
$$[M,\Lambda]=[M,\Lambda_{1,m}]
=[M,\Lambda_{2,m}]=
[M,\Lambda_{1,2}]=
[M,\Lambda_{2,2}]=0_n.$$ 
Moreover, this condition is satisfied if
$\Lambda$,
$\Lambda_{1,m}$ and $\Lambda_{2,m}$ for $m\ge4$,
as well as
$\Lambda_{1,2}$ and $\Lambda_{2,2}$,
are entry-wise diagonal matrices
whose off-diagonal entries all vanish. 
Throughout this section, we shall use the term
``diagonal'' in this sense.
Motivated by this observation, 
we investigate the conditions on the nonlinear terms 
that ensure this property. 
In what follows, 
we impose (A1)--(A2). 
\par 
First, we investigate when 
$\Lambda(Q)$ 
is diagonal in the above sense. 
By \eqref{eq:Lambda} and \eqref{eq:Rh}, 
we see 
$\Lambda(Q)$ is diagonal if and only if 
$P_2(Q)+P_2^{\top}(Q)$ is so.  
By \eqref{eq:P2}, we have  
\[
P_2(Q)+P_2^{\top}(Q)
=
\left(
\sum_{p,q=1}^n
(\omega^{2,j}_{r,q,p}+\omega^{2,r}_{j,q,p})
\, Q_p\, Q_q
\right)_{1\leqslant j,r\leqslant n}. 
\]
Hence, $\Lambda(Q)$ is diagonal if the condition (E1) is satisfied. 
\par 
Second, we investigate when 
$\Lambda_{1,m}(Q,\p_xQ)$ and 
$\Lambda_{1,2}(Q,\p_xQ)$ 
are diagonal. 
By \eqref{eq:Lambda1m} and \eqref{eq:KmSoff},
both $\Lambda_{1,m}$ and $\Lambda_{1,2}$ are diagonal
if and only if
\[
B_m-B_m^{\top}-\p_x(P_2-P_2^{\top})
\]
and the corresponding expression with $m=2$
are diagonal.
Since both matrices are skew-symmetric,
they are diagonal if and only if they vanish identically.
As discussed in Section~\ref{sec:LWP4}, 
the conditions (D1)--(D2) under (A1)
ensure that this matrix vanishes identically.
Note that imposing additional conditions (D1)--(D2) under (A1) 
eliminates the requirement of $m\geqslant 5$ as is also  discussed in Section~\ref{sec:LWP4}. 
\par
Third, we investigate when 
$\Lambda_{2,m}(Q,\p_xQ)$ and 
$\Lambda_{2,2}(Q,\p_xQ)$ 
are diagonal. 
By 
\eqref{eq:Lambda2m}, 
both of them are diagonal if and only if 
both $D_m(Q)$ with $m\geqslant 4$ and $D_2(Q)$ 
are so. 
Recall that each $D_m(Q)$ is specified to satisfy 
the condition ($\widetilde{C2}$).  
By the definition 
\eqref{eq:K_m} of $\mathcal{K}_m(Q,\p_xQ)$, 
it follows that  
\begin{align}
\p_x(D_m(Q))
&=
\begin{pmatrix}
	I_n & 0_n
\end{pmatrix}
\mathcal{R}_{m,S}^{\mathrm{diag}}(Q,\p_xQ)
\begin{pmatrix}
	I_n & 0_n
\end{pmatrix}^{\top}
\nonumber
\\
&\quad 
-\p_x(
\begin{pmatrix}
	I_n & 0_n
\end{pmatrix}
\mathcal{R}_{S}^{\mathrm{diag}}(Q)
\begin{pmatrix}
	I_n & 0_n
\end{pmatrix}^{\top}
).
\nonumber
\end{align}
The first term of the right-hand side is given by the 
spatial derivative of \eqref{eq:RmSdiagp}. 
The second term is given by the spatial derivative 
of $\frac{1}{2}(P_1(Q)-P_1^{\star}(Q))$, 
since $\mathcal{R}_S=\frac{1}{2}(\mathcal{R}-\mathcal{R}^*)$. 
Noting them and using \eqref{eq:P1} on the expression of $P_1$, we obtain   
\begin{align}
D_m(Q)
&=
\dfrac{1}{2}
\,\sum_{p,q=1}^n
\left\{
m
(\omega^{1,j}_{r,q,p}-
\overline{\omega^{1,r}_{j,p,q}}
)
+(\omega^{3,j}_{r,q,p}
-2\overline{\omega^{4,r}_{j,q,p}}
)
\right\}
\,  
Q_p\, \overline{Q_q}
\nonumber
\\
&\quad 
-\dfrac{1}{2}
\sum_{p,q=1}^n
(\omega_{r,q,p}^{1,j}
-\overline{\omega_{j,p,q}^{1,r}}
)
Q_p\overline{Q_q}
\nonumber
\\
&=
\dfrac{1}{2}
\,\sum_{p,q=1}^n
\left\{
(m-1)
(\omega^{1,j}_{r,q,p}-
\overline{\omega^{1,r}_{j,p,q}}
)
+(\omega^{3,j}_{r,q,p}
-2\overline{\omega^{4,r}_{j,q,p}}
)
\right\}
\,  
Q_p\, \overline{Q_q}. 
\nonumber
\end{align}
This shows that both $D_m(Q)$ with $m\geqslant 4$ 
and $D_2(Q)$
are diagonal if the conditions 
(E2) and (E3) are satisfied. 
\par 
Putting these observations together, 
even if 
$M=\diag(a_1,\ldots,a_n)$ for 
$(a_1,\ldots,a_n) \in (\mathbb{R} \setminus \{0\})^n$,  
the local well-posedness 
holds in $H^m$ for any integer $m\geqslant 4$ 
under the conditions (A1)--(A2), (C1)--(C2), (D1)--(D2) and 
additional (E1)--(E3). 
\end{proof}
\begin{remark}
The case
$M=\diag(a_1,\ldots,a_n)$ for
$(a_1,\ldots,a_n)\in (\mathbb{R}\setminus\{0\})^n$
was treated in \cite{Onodera2025} from the same perspective,
where local well-posedness in $H^m$
for integers $m\geqslant 4$
was proved under
(A1)--(A2), (B1)--(B6),
together with the additional conditions
(B7)--(B9) given by
\begin{align}
\tag{B7}
\omega_{k,q,r}^{1,j}
=0
\quad &\text{unless $j=k$ for all} \quad q,r,j,k\in 
\{1,\ldots,n\},
\\
	\tag{B8}
	\omega_{k,q,r}^{2,j}
	=0
	\quad &\text{unless $j=k$ for all} \quad q,r,j,k\in 
	\{1,\ldots,n\},
\\
	\tag{B9}
	\omega_{k,q,r}^{3,j}+2\omega_{k,r,q}^{4,j}
	=0
	\quad &\text{unless $j=k$ for all} \quad q,r,j,k\in 
	\{1,\ldots,n\}.
\end{align}
The set of conditions (A1)--(A2), (C1)--(C2), (D1)--(D2), and (E1)--(E3) 
is weaker than the alternative set of conditions (A1)--(A2), (B1)--(B6), and (B7)--(B9).  
Indeed, the conditions (E1)--(E3) are directly derived from the latter set: 
(E1) from (B8), (E2) from (B7), and (E3) from a combination of (B6) and (B9).
Conversely, conditions such as (B6) and (B9) cannot be derived solely from the former set.
\end{remark}
\section{Proof of Theorem~\ref{thm:main2}}
\label{sec:G}
In this section, we  present the proof of Theorem~\ref{thm:main2}. 
\begin{proof}[Proof of Theorem~\ref{thm:main2}] 
The argument follows the same strategy 
as in the proof of Theorem~\ref{thm:main}, 
by verifying that the additional nonlinear terms 
preserve the key structural properties.
For simplicity, we assume $F(Q,\p_xQ,\p_x^2Q)=0$ and 
focus only on the conditions on $G(Q,\p_xQ,\p_x^2Q)$. 
It is straightforward to see 
that the set of conditions (A2) and (C1)--(C2) on $F$ 
which is the same imposed in Theorem~\ref{thm:main}
and the additional conditions on $G$ 
do not interact with each other. 
\par  
In view of Proposition~\ref{pro:framework} proved in
Section~\ref{sec:proof} and the proof of
Theorem~\ref{thm:main} in Section~\ref{sec:A2C1C2},
the proof of Theorem~\ref{thm:main2} reduces to verifying
($\widetilde{C1}$)-($\widetilde{C2}$) for the problem
\eqref{eq:IVP} with
$F(Q,\p_xQ,\p_x^2Q)$ replaced by
$G(Q,\p_xQ,\p_x^2Q)$.
More precisely, it suffices to show that the structural
conditions (A3)--(A5) and (C3)--(C4)
imply ($\widetilde{C1}$)-($\widetilde{C2}$)
for $m\geqslant5$ and $m=2$.
In what follows, we impose
(A3)--(A5) and (C3)--(C4).
\par 
For this purpose, we use \eqref{eq:U_eq}
to determine $P_1(Q)$ and $A_m(Q,\p_xQ)$,
and then verify ($\widetilde{C1}$)-($\widetilde{C2}$)
for $Q:I_T\times\TT\to\CC^n$.
\par   
First, since $(P_1(Q)\p_x^2U)_j$, which denotes the 
$j$-th component of $P_1(Q)\p_x^2U$,  
is equal to the sum of the terms including $\p_x^2U$ 
in $\p_x^m\{F_j(Q,\p_xQ,\p_x^2Q)\}$, 
we have
\begin{align}
(P_1(Q)\p_x^2U)_j
&=
\sum_{p,q,r=1}^n
\mu_{p,q,r}^{1,j}\,\p_x^2U_p\,Q_q\,Q_r
+
\sum_{p,q,r=1}^n
\mu_{p,q,r}^{3,j}\,\p_x^2U_p\,\overline{Q_q}
\,\overline{Q_r}
\nonumber
\\
&=
\sum_{p,q,r=1}^n
\mu_{r,q,p}^{1,j}\,Q_p\,Q_q\,\p_x^2U_r
+
\sum_{p,q,r=1}^n
\mu_{r,q,p}^{3,j}
\,\overline{Q_p}
\,\overline{Q_q}\,\p_x^2U_r.
\nonumber
\end{align}
From this, it follows that  
\begin{align}
&P_1(Q)
=
\left(
\sum_{p,q=1}^n
\mu_{r,q,p}^{1,j}\,Q_p\,Q_q
+
\sum_{p,q=1}^n
\mu_{r,q,p}^{3,j}
\,\overline{Q_p}
\,\overline{Q_q}
\right)_{1\leqslant j,r\leqslant n}, 
\nonumber
\\
&P_1(Q)+P_1^*(Q)
\nonumber
\\
&=
\left(
\sum_{p,q=1}^n
\left(
\mu_{r,q,p}^{1,j}
+
\overline{\mu_{j,q,p}^{3,r}}
\right)
Q_pQ_q
+
\sum_{p,q=1}^n
\left(
\overline{\mu_{j,q,p}^{1,r}}
+
\mu_{r,q,p}^{3,j}
\right)
\overline{Q_p}\,\overline{Q_q}
\right)_{1\leqslant j,r\leqslant n}.   
\nonumber
\end{align} 
By the relation \eqref{eq:imp1}, 
the condition  
($\widetilde{C1}$) 
holds if and only if $0_n=P_1(Q)+P_1^*(Q)$.
Hence, using the symmetry conditions
(A3) and (A5) to interchange the indices
$p$ and $q$, we obtain
$P_1(Q)+P_1^*(Q)=0_n$
from (C3).
Therefore, ($\widetilde{C1}$) holds.
\par 
Second, 
since $(A_m(Q,\p_xQ)\p_xU)_j$ equals the sum of 
the terms in $\p_x^m\left\{F_j(Q,\p_xQ,\p_x^2Q)\right\}$ 
which include $\p_xU$, 
\begin{align}
(A_m(Q,\p_xQ)\p_xU)_j
&=
m\sum_{p,q,r=1}^n
\mu_{p,q,r}^{1,j}
\,\p_xU_p
\,\p_x(Q_qQ_r)
+
\sum_{p,q,r=1}^n
\mu_{p,q,r}^{2,j}
\,\p_xU_p
\,\p_xQ_q
\,Q_r
\nonumber
\\
&\quad 
+
\sum_{p,q,r=1}^n
\mu_{p,q,r}^{2,j}
\,\p_xQ_p
\,\p_xU_q
\,Q_r
+
m
\sum_{p,q,r=1}^n
\mu_{p,q,r}^{3,j}
\,\p_xU_p
\,\p_x(\overline{Q_q}\,\overline{Q_r})
\nonumber
\\
&\quad 
+
\sum_{p,q,r=1}^n
\mu_{p,q,r}^{6,j}
\,\p_xU_p
\,\overline{\p_xQ_q}\,\overline{Q_r}
\nonumber
\\
&=
m\sum_{p,q,r=1}^n
\mu_{r,q,p}^{1,j}
\,\p_x(Q_pQ_q)
\,\p_xU_r
\nonumber
\\
&\quad 
+
\sum_{p,q,r=1}^n
(\mu_{r,p,q}^{2,j}+\mu_{p,r,q}^{2,j})
\,\p_xQ_p
\,Q_q
\,\p_xU_r
\nonumber
\\
&\quad 
+
m
\sum_{p,q,r=1}^n
\mu_{r,q,p}^{3,j}
\,\p_x(\overline{Q_p}\,\overline{Q_q})
\,\p_xU_r
+
\sum_{p,q,r=1}^n
\mu_{r,p,q}^{6,j}
\,\overline{\p_xQ_p}
\,\overline{Q_q}
\,\p_xU_r.
\nonumber
\end{align}
Thus the $(j,r)$-component of $A_m(Q,\p_xQ)$  
is given by 
\begin{align}
(A_m(Q,\p_xQ))_{jr}
&=
\p_x\left(
\sum_{p,q=1}^n
m
\left(
\mu_{r,q,p}^{1,j}
\,Q_pQ_q+
\mu_{r,q,p}^{3,j}
\overline{Q_p}\,\overline{Q_q}
\right)
\right)
\nonumber
\\
&\quad 
+
\sum_{p,q=1}^n
(\mu_{r,p,q}^{2,j}+\mu_{p,r,q}^{2,j})
\,\p_xQ_p
\,Q_q
+
\sum_{p,q=1}^n
\mu_{r,p,q}^{6,j}
\,\overline{\p_xQ_p}
\,\overline{Q_q}. 
\nonumber
\end{align}
Therefore, a simple computation yields 
\begin{align}
&(A_m(Q,\p_xQ)-A_m^{*}(Q,\p_xQ))_{jr}
\nonumber
\\
&=
\p_x\left(
\sum_{p,q=1}^n
m
\left(
\left(\mu_{r,q,p}^{1,j}
-
\overline{\mu_{j,q,p}^{3,r}}
\right)
\,Q_pQ_q+
\left(\mu_{r,q,p}^{3,j}
-
\overline{\mu_{j,q,p}^{1,r}}
\right)
\overline{Q_p}\,\overline{Q_q}
\right)
\right)
\nonumber
\\
&\quad 
+
\sum_{p,q=1}^n
\left(
\mu_{r,p,q}^{2,j}+\mu_{p,r,q}^{2,j}
-\overline{
\mu_{j,p,q}^{6,r}
}
\right)
\,\p_xQ_p
\,Q_q
-
\sum_{p,q=1}^n
\left(
\overline{\mu_{j,p,q}^{2,r}}
+
\overline{\mu_{p,j,q}^{2,r}}
-\mu_{r,p,q}^{6,j}
\right)
\,\overline{\p_xQ_p}
\,\overline{Q_q}. 
\nonumber
\end{align}
Here, (A4) and (C4) ensure 
$\mu_{r,p,q}^{2,j}+\mu_{p,r,q}^{2,j}
-\overline{
\mu_{j,p,q}^{6,r}}=2\mu_{r,p,q}^{2,j}
-\overline{
\mu_{j,p,q}^{6,r}}$
and the right-hand side is invariant under the change 
of $p$ and $q$. 
By this reason,    
\[
\sum_{p,q=1}^n
\left(
\mu_{r,p,q}^{2,j}+\mu_{p,r,q}^{2,j}
-\overline{
\mu_{j,p,q}^{6,r}
}
\right)
\,\p_xQ_p
\,Q_q
=
\p_x\left(
\dfrac{1}{2}
\sum_{p,q=1}^n
\left(
2\mu_{r,p,q}^{2,j}
-\overline{
\mu_{j,p,q}^{6,r}}
\right)
\,Q_p\,Q_q
\right),  
\]
\[
\sum_{p,q=1}^n
\left(
\overline{\mu_{j,p,q}^{2,r}}
+
\overline{\mu_{p,j,q}^{2,r}}
-\mu_{r,p,q}^{6,j}
\right)
\,\overline{\p_xQ_p}
\,\overline{Q_q}
=
\p_x\left(
\dfrac{1}{2}
\sum_{p,q=1}^n
\left(
2\overline{\mu_{j,p,q}^{2,r}}
-\mu_{r,p,q}^{6,j}
\right)
\,\overline{Q_p}\,\overline{Q_q}
\right). 
\]
Combining them, and recalling the relation \eqref{eq:6181}, 
we can express 
\[
\begin{pmatrix}
	I_n & 0_n
\end{pmatrix}
\mathcal{R}_{m,S}^{\mathrm{diag}}(Q,\p_xQ)
\begin{pmatrix}
	I_n & 0_n
\end{pmatrix}^{\top}
=\p_x\left(
\widetilde{D}_m(Q)
\right),  
\]
where $\widetilde{D}_m(Q)$ is an $n\times n$ matrix-valued 
function on $I_T\times \TT$,   
and each component is of order $|Q|^2$. 
This shows that ($\widetilde{C2}$) with 
$m\geqslant 5$ and $m=2$ is satisfied in this setting.
\end{proof}
\section*{Acknowledgements}
The author has been supported by JSPS Grant-in-Aid for Scientific Research (C) 
Grant Number JP24K06813.
%


\end{document}